\documentclass{article}
\usepackage{amsthm}
\usepackage{graphicx} 
\usepackage{listings}
\usepackage[colorlinks=true, urlcolor=blue, linkcolor=blue, citecolor=blue]{hyperref}
\usepackage{geometry}
\usepackage[utf8]{inputenc}

\usepackage{makeidx}
\usepackage{latexsym}
\usepackage{amsfonts}
\usepackage{amssymb}
\usepackage{amsmath}
\usepackage{amstext}
\usepackage{amsthm}
\usepackage{color}
\usepackage{mathrsfs}
\usepackage{relsize}
\usepackage{subfig}
\usepackage{mathtools}
\usepackage{graphicx}
\usepackage{sidecap}
\usepackage{wrapfig}
\usepackage[dvipsnames]{xcolor}
\usepackage{pstricks}

\graphicspath{ {./images/} }
\usepackage[english]{babel}
\usepackage[T1]{fontenc}
\usepackage{cancel}
\usepackage{amsmath}
\usepackage{subfig}
\usepackage{braket}
\usepackage{mathrsfs}
\usepackage{scalerel,stackengine}
\stackMath
\usepackage{comment}
\newcommand\reallywidehat[1]{%
\savestack{\tmpbox}{\stretchto{%
  \scaleto{%
    \scalerel*[\widthof{\ensuremath{#1}}]{\kern-.6pt\bigwedge\kern-.6pt}%
    {\rule[-\textheight/2]{1ex}{\textheight}}
  }{\textheight}%
}{0.5ex}}%
\stackon[1pt]{#1}{\tmpbox}%
}
\usepackage{tikz}
\usetikzlibrary{arrows.meta,decorations.markings}
\usetikzlibrary{positioning}
\makeatletter
\providecommand*{\cupdot}{%
  \mathbin{%
    \mathpalette\@cupdot{}%
  }%
}
\newcommand*{\@cupdot}[2]{%
  \ooalign{%
    $\m@th#1\cup$\cr
    \hidewidth$\m@th#1\cdot$\hidewidth
  }%
}
\makeatother

\usepackage{amsthm}
\newtheorem{theorem}{Theorem}
\newtheorem{defn}[theorem]{Definition}

\newtheorem{prop}[theorem]{Proposition}
\newtheorem{cor}[theorem]{Corollary}
\newtheorem{remark}[theorem]{Remark}
\newtheorem{oss}[theorem]{Remark}
\newtheorem{example}[theorem]{Example}
\newtheorem{lemma}[theorem]{Lemma}

\newtheorem{corollary}[theorem]{Corollary} 

\usepackage{blindtext}
\usepackage{amssymb}
\usepackage{enumitem}
\usepackage{hyperref}
\hypersetup{
colorlinks=true,
linkcolor=blue,
filecolor=magenta,
urlcolor=cyan,
}

\newcommand{\J}[1]{\textcolor{red}{#1}}

\newcommand{\D}[1]{\textcolor{blue}{#1}}

\newcommand{\coloneq}{\mathrel{\mathop:}=}

\title{The neighbourhood convexity} 
\author{Daniela Bubboloni~\thanks{Corresponding author.}~\footnote{Dipartimento di Matematica e Informatica ``Ulisse Dini'', Universit\`{a} degli Studi di Firenze, viale Morgagni, 67/a, 50134 Firenze, (Italy), {\tt daniela.bubboloni@unifi.it}, \href{https://orcid.org/0000-0002-1639-9525}{orcid.org/0000-0002-1639-9525}.} \and José Cáceres~\footnote{CDTIME and Departamento de Matem\'{a}ticas, Universidad de Almería, ctra. Sacramento s/n, 04120 Almería (Spain), {\tt jcaceres@ual.es}, \href{https://orcid.org/0000-0003-2790-8484}{orcid.org/0000-0003-2790-8484}. }}
\date{}

\begin{document}

\maketitle
\begin{abstract}
In this paper, we investigate the neighbourhood convexity ($n$-convexity) on graphs, a new finite convexity space grounded in the common closed neighbourhood closure operator. Unlike standard path-based graph convexities, $n$-convexity shows a non-canonical behaviour, giving rise to compelling structural properties and being almost never hereditary. Focusing on the properties of graphs that form $n$-convex geometries, a parity distinction emerges: an $n$-convex geometry contains a star vertex if and only if the number of its vertices is odd. Every odd-order $n$-convex geometry can be uniquely constructed by attaching a star vertex to an even-order one. We introduce the concept of quasi-stars (vertices of degree $\vert{}V\vert{}-2$) and prove a reduction property that allows systematically reducing an $n$-convex geometry by removing a pair of vertices, one of which is a quasi-star.  Finally, we explore the connections between $n$-convexity and $P(G)$, the neighbourhood preorder, demonstrating that $n$-convex sets are upsets of $P(G)$ and that, in star-free $n$-convex geometries, quasi-stars correspond precisely to the maximal elements of $P(G)$. We complete our study by classifying quasi-threshold and threshold $n$-convex geometries.
\end{abstract}
\section{Introduction }
\label{sec:introduction}

While initially developed in Analysis and Geometry, the notion of convexity has been fruitful in many areas of mathematics, such as Mathematical Programming~\cite{ly16}, Computational Geometry~\cite{ps85}, Probability~\cite{bp05}, Game Theory~\cite{be00} and Economics~\cite{m16}. In particular, finite convexities have exerted a significant influence on Graph Theory~\cite{d87}, as they allow for the characterization of many classes of 
graphs, such as chordal, ptolemaic graphs~\cite{fj86} and HHD-free graphs~\cite{dnb99,co09} (see~\cite{dgpst25} for a survey). Beyond these characterizations, this framework has driven considerable computational interest in efficiently calculating the elements of different convexities~\cite{dgkps09}. Furthermore, it has established deep connections between convexities and other structural graph parameters, while also inspiring the 
definition of novel graph parameters rooted in convex concepts such as the hull number, the interval number and the interval iteration number for the usual convexities that can be defined in a graph (see~\cite{p13} for a survey of these diverse research lines).

In this paper, we introduce a new convexity in graphs that we call neighbourhood convexity or $n$-convexity. A non-empty set $K$ of vertices of a graph $G$ is $n$-convex if there exists a set $Y$ of vertices such that $K$ is formed by those vertices of $G$ which are adjacent to every vertex of $Y.$  This convexity, as far as we know, has never been addressed, and it is not based, as usual, on paths among vertices of the graph but on a Moore closure operator, introduced by Bubboloni and Pinzauti (\cite{1020}) in order to deal with the reconstruction of the directed power graph of a finite group from its undirected counterpart. The role of this operator in successfully addressing other related group theoretical questions (\cite{bp25a}) invites to explore its properties from a pure graph theoretical point of view. 

A fundamental problem in graph convexity is determining whether a convexity constitutes a convex geometry. That study is largely motivated by the compact representation of convex sets, as every convex set is uniquely determined and efficiently encoded by its minimal generating set of extreme points (see~\cite{vv93}).
In our research, we investigate $n$-convex geometries and focus on its structural properties. Remarkably, the neighbourhood convexity is almost never hereditary, that is, the fact that a graph is an $n$-convex geometry does not guarantee that every induced subgraph of it is still an $n$-convex geometry (Theorem~\ref{nohered}).
Since, as pointed out in \cite{dgpst25}, very few non-hereditary convex geometries are known, our research can be seen as a contribution in this direction.
Specifically, we show that  a convex geometry cannot admit true twins (Theorem~\ref{noclosedtwins}) and thus it admits at most one star vertex. Remarkably,  even and odd-order convex geometries bifurcate with respect to the existence of stars. Indeed we establish that an $n$-convex geometry admits a star if and only if the number of its vertices is odd (Corollary~\ref{evenodd}). As a consequence,
any odd-order convex geometry can be constructed from an even-order one by attaching a star-like vertex (Proposition~\ref{charodd}). On the contrary, even if one construction  to pass from an odd-order $n$-convex geometry to an even-order one is known, it remains an open question how to construct any even-order convex geometry (see Section~\ref{sec:final}).

A key concept for establishing some of the above mentioned results and many other throughout the paper, is the one of quasi-star, that is, a vertex which is adjacent to any vertex up to one.
 Our structural analysis reveals a bijection between the set of extreme points of an $n$-convex geometry and the set $\mathcal{D}$ of quasi-stars (Theorem~\ref{stars-extreme}). $\mathcal{D}$ is also at the base of a main reduction property, showing that it is always possible to reduce a given convex geometry with no star to a smaller one by removing two vertices, one of them belonging to $\mathcal{D}$. 

Recent years have witnessed renewed interest in convex geometries from several perspectives. Contributions include the study of different notions of dimension for convex geometries \cite{KnauerTrotter2024}, geometric realizability in Euclidean spaces, ideals of convex geometries and their connections with complexes of oriented matroids \cite{cck24}, as well as special subclasses such as supersolvable convex geometries, defined as a generalization of the lattice of subgroups in a finite supersolvable group, and convex geometries arising from building sets \cite{BackmanDanner2024}. These developments highlight the richness of the theory and its interactions with order theory, lattice theory, discrete geometry and algebra \cite{Arm, Ma}. In the present work, a connection with order theory  naturally emerges through the neighbourhood preorder $P(G)$ of a graph $G$. Such a preorder, simply prescribes, for $x,y\in V(G)$, that $x\leq y$ if the closed  neighbourhood of $x$ is included in the  closed  neighbourhood of $y$ (see Section \ref{sec:pre}). $P(G)$ turns out to encode significant information on the neighbourhood convexity. 
Indeed,  $n$-convexity refines  the upsets convexity of $P(G)$, upsets with a unique minimal  element are $n$-convex and the extreme vertices of an $n$-convex set are a subset of the minimal elements of $P(G)$ (Lemma~\ref{upset1}).  
For an $n$-convexity, $P(G)$ becomes a poset and its maximal elements are the quasi-stars (Proposition~\ref{Dposet}).  

We emphasize that the neighbourhood structure of a graph has been considered in many papers. For instance, in \cite{bst93}, relations of various kind on the vertex set as well as graphs and directed or mixed graphs are associated with the neighbourhood structure of a graph.
In \cite{Foldes}, it is considered, for a graph $G$, the so-called vicinal preorder in which, for $x,y\in V(G)$, $x\preceq y$ means that the neighbourhood of the vertex $x$ is contained in the closed neighbourhood of the vertex $y.$ Such a preorder is at the base of the definition of the Dilworth number  and has been widely considered in the literature (see, for instance, \cite{Peled},\cite{Li}). 
Since $P(G)$ refines the vicinal preorder one can ask if every $n$-convex set is indeed an upset for the vicinal preorder. However, it is easily seen that this is not the case (see the end of Section \ref{sec:pre}). In other words our framework specifically requires a variation of  the classical concepts in the literature.

The paper is organized as follows: Section 2 provides preliminaries, while Section~\ref{sec:neighbourhood} introduces neighbourhood convexity. Section~\ref{sec:convexgeometry} presents $n$-convex geometries and quasi-star vertices. Section~\ref{sec:preorder} establishes the connection between $n$-convexity and neighbourhood preorders. Section~\ref{construction} describes the star-augmented graph, and Section~\ref{sec:stars} characterizes stars using vertex parity. Section~\ref{sec:applications} provides applications, including threshold $n$-convex geometries. Finally, Section~\ref{sec:final} introduces quasi-star-augmented graphs, followed by a section of conclusions and future lines of work.


\section{Preliminaries}
\label{sec:preliminaries}
We denote by $\mathbb{N}$ the set of positive integers and, for  $k\in \mathbb{N}\cup \{0\}$, we set $[k]\coloneq\{x\in \mathbb{N}: x\leq k\}$, $[k]_0\coloneq\{x\in \mathbb{N}\cup \{0\}: x\leq k\}$. 
For a set $X$, we use the notation  $X^{(2)}\coloneq\{A\subseteq X:|A|=2\};$  given   $Y, Z\subseteq X$, we write $X=Y\cupdot Z$ to mean $X=Y\cup Z$ and $Y\cap Z=\varnothing.$ 

\subsection{Graphs}
All graphs in this paper are finite and undirected. 
Thus a graph $G$ is a pair $(V,E),$ where $V$ is a finite non-empty set and $E\subseteq V^{(2)}.$ 
Let $n,m\in \mathbb{N}.$  The \textbf{path} on $n$ vertices, is the graph $P_n$ with vertex set $[n]$ and edges 
$\{i,i+1\}$ for $i\in [n-1]$; the \textbf{cycle} on $n$ vertices 
 is the graph $C_n$ with vertex set $[n]$ and  edges $\{i,i+1\}$ for $i\in [n-1]$ and $\{n,1\}$. As usual we denote by $K_n$ the {\bf complete graph} on $n$ vertices and by $K_{m,n}$ the {\bf complete bipartite graph} with parts of sizes $m$ and $n$.
For a given family $\mathcal{X}$ of graphs, we say that a graph $G$ is \textbf{$\mathcal{X}$-free} if no induced subgraph of $G$ is isomorphic to a graph in the family $\mathcal{X}$.

Let $G=(V,E)$ be a graph and $x\in V$. We refer to $|V|$ as the {\bf order} of $G.$  The {\bf open neighbourhood} of $x$ is $N(x)=\{y\in G: \{x,y\}\in E\}$ and the {\bf closed neighbourhood} of $x$ is $N[x]=N(x)\cup\{x\}$. The subgraph induced by $X\subseteq V$ is denoted by $G[X]$. If $H$ is a subgraph of $G$, we write $H\subseteq G$; if $H$ is an induced subgraph of $G$, we write $H\leq G.$
A non-empty subset $X$ of $V$ is called a {\bf clique} if $G[X]$ is a complete graph.
Two vertices $x,y\in V$ are called {\bf true twins} if $N[x]=N[y]$; {\bf false twins} if $N(x)=N(y)$. 
Two distinct vertices which are true [false] twins are said to be  a pair of true [false] twins. $G$ is called {\bf true twins free} if $G$ does not admit pairs of true twins. 
The relation of being closed twins 
is an equivalence relation. We denote the equivalence class of $x$, with respect to the true twin relation, by $[x]_N.$
The connected component of $G$ containing the vertex $x\in V$ is denoted by $C(x)$. In any notation, we adopt subscripts denoting the graph under consideration only if misunderstanding  is possible.
For undefined basic concepts we refer the reader to introductory graph theoretical literature, e.g.,~\cite{w99}.

\subsection{Convexities}
We recall the definitions about convexity in the context of finite sets which we are going to use throughout the paper.

 A family $\mathscr{C}$ of subsets of a finite  set $V$ is called a \textbf{convexity} in $V$, or we say that $(V,\mathscr{C})$ is a \textbf{convexity space} if $\varnothing, V\in \mathscr{C}$, and $X,Y\in\mathscr{C}$ implies $X\cap Y\in \mathscr{C}.$
The elements in $\mathscr{C}$ are called {\bf $\mathscr{C}$-convex} or more simply {\bf convex} if the convexity is clear from the context. The convexity $2^V$ is the  {\bf free convexity}; the convexity $\{\varnothing, V\}$ is the {\bf coarse convexity}. Given  $X\subseteq V$, the  \textbf{convex hull} of $X$, denoted by $[X]_{\mathscr{C}}$, is defined by
$$[X]_{\mathscr{C}}\coloneq \bigcap\{C\in \mathscr{C}: X\subseteq C\}.$$
Clearly, $[X]_{\mathscr{C}}$ is the minimum $\mathscr{C}$-convex set containing $X$. In particular, $X=[X]_{\mathscr{C}}$ if and only if $X$ is $\mathscr{C}$-convex.
Given $X\in \mathscr{C}$, an element $x\in X$ is called an {\bf extreme} of $X$ if $X- \{x\}\in \mathscr{C}.$ The set of extreme points of $X$ is denoted by $\mathrm{ex}_{\mathscr{C}}(X)$. Given $v \in V$, a convex set $C \in \mathscr{C}$ 
is said to be a $\mathscr{C}$-{\bf copoint} attached at $v$ if
 $C \subseteq V-\{v\}$  and it is maximal among the convex sets avoiding $v$, that is, for every $D \supsetneq C$ such that $D \in \mathscr{C}$, we have $v \in D$.

A convexity space $(V,\mathscr{C})$ is called a {\bf convex geometry} if, for every $K\in\mathscr{C}$, $K=[\mathrm{ex}_{\mathscr{C}}(K)]_{\mathscr{C}}$. 
Clearly, the free convexity is always a convex geometry.
Convex geometries mimic in the discrete framework some main properties of the classical convexity in the Euclidean space $\mathbb{R}^n$. One main property is given by the anti-exchange property.
 A convexity space $(V,\mathscr{C})$ satisfies the {\bf anti-exchange property} if, for every $K \in \mathscr{C}$ and $v,w\in  V-K$, $v\in [K \cup \{w\}]_{\mathscr{C}}$ and $w\in [K \cup \{v\}]_{\mathscr{C}}$ imply $v=w.$
  
\begin{theorem}{\rm (\cite{Jamison, vv93})} 
        \label{62}
        Let $(V,\mathscr{C})$ be a convexity space. Then the following facts are equivalent:
        \begin{enumerate} [label=$(\arabic*)$, ref=\thetheorem (\arabic*)]
        \item \label{62first} $(V,\mathscr{C})$ is a convex geometry.
            \item \label{62second} $(V,\mathscr{C})$ has the anti-exchange property.
            \item \label{62third} For every $K \in \mathscr{C}-\{V\}$, there exists $v \in V-K$ such that $K 	\cup \{v\} \in \mathscr{C}$.
            \item \label{62forth} It does not exists a $\mathscr{C}$-copoint attached at two different points.
        \end{enumerate}
\end{theorem}
We present now a result of general interest about copoints.
\begin{lemma}
\label{202}
Let $(V,\mathscr{C})$ be a convex geometry,  $X\subseteq V$ and $K \in \mathscr{C}$ be such that $X\not\subseteq K$. Then, there exist $x\in X$ and a $\mathscr{C}$-copoint $D$ attached at $x$ such that $(X-\{x\}) \cup K \subseteq D$.
\end{lemma}
\begin{proof}
 First, we observe that for every $x \in X-K\neq\varnothing$, there exists at least a $\mathscr{C}$-copoint $D$ attached at $x$ which contains $K$. Indeed such $D$ is a maximal element, with respect to inclusion, of the non-empty finite set $\{D' \in \mathscr{C}: K\subseteq D'\subseteq V-\{x\} \}$. So the set $$\mathcal{L}\coloneq\{ D\subseteq V : K\subseteq D \mbox { such that} \ D \ \mbox{is a} \ \mathscr{C}\mbox{-copoint attached at} \ x \mbox{ for some } x\in X\}$$ is non-empty and, consequently, $\ell\coloneq\max\{|D \cap X|:D\in\mathcal{L}\}$ is well defined.
 
 Let $D^*\in\mathcal{L}$ realizing $\ell$. Then, $D^*$ contains $K$ and is a $\mathscr{C}$-copoint attached at a certain $x^*\in X$. Since $D^* \subseteq V-\{x^*\}$, then we have $|D^* \cap X| \leq |X|-1$. Note also that we have $|(D^* \cup \{x^*\}) \cap X|=|D^* \cap X|+1$. Then for every $x\in X$, we have that $D^* \cup \{x^*\}$ is not a $\mathscr{C}$-copoint attached at $x$. Now we show that $X \subseteq D^* \cup \{x^*\}$. Suppose, by absurd, that there exists $x\in X$ such that $x \not \in D^* \cup \{x^*\}$. We clearly have that $x\neq x^*$ and $K \subseteq D^* \cup \{x^*\} \subseteq V-\{x\}$. Since $(G,\mathscr{C})$ is a convex geometry and $D^*$ is a $\mathscr{C}$-copoint attached at $x^*$, then, by Theorem~\ref{62third}, $D^* \cup \{x^*\}$ is a $\mathscr{C}$- convex set. Then, there exists a $\mathscr{C}$-copoint $D$ attached at $x$ such that $D^* \cup \{x^*\} \subseteq D$. Thus we have $|D \cap X| \geq |(D^* \cup \{x^*\}) \cap X|=|D^* \cap X|+1,$ contradicting the fact that $D^*$ is realizing $\ell$. Thus $X \subseteq D^* \cup \{x^*\}$ and so $(X-\{x^*\}) \cup K \subseteq D^*$.
\end{proof}

For details on those and other definitions about convexities, we refer the reader to~\cite{vv93}.

\section{Neighbourhood convexity and its properties}
\label{sec:neighbourhood}

In this section, we introduce the neighbourhood convexity and give some of its main properties.
We begin with the definition of two operators: the common closed neighbourhood and the neighbourhood closure.

\begin{defn}[\cite{1020}]\label{def:main}
{\rm Let $G=(V,E)$ be a graph and let $X \subseteq V$. The {\bf common closed neighbourhood} of $X$ in $G$ is defined by
$$N[X]\coloneqq\begin{cases}
\displaystyle{\bigcap_{x\in X}  N[x]}&{\text{\rm if }} X\neq \varnothing \\
&\\
V&{\text{\rm if }} X= \varnothing.
\end{cases}$$
 The vertices in $\mathcal{S}\coloneq N[V]$ are called the {\bf stars} of the graph $G$ and $\mathcal{S}$ is called the {\bf star set} of $G.$}
\end{defn}

 Note that $\mathcal{S}=V$ if and only if $G$ is a complete graph and that, for every $X\subseteq V$, and every $S\subseteq \mathcal{S}$, we have $N[X]=N[X- S]$. The sets $X$ and $N[X]$ are typically incomparable by inclusion. Consider, for instance, $G\coloneq C_4$  and $X\coloneq\{1,3\}$. Then $N[X]=\{2,4\}$ trivially intersects $X$.
 
We indicate by $\mathcal{N}$ the {\bf family of common closed neighbourhoods} of subsets of $V$, that is, we set $$\mathcal{N}:=\{N[Y] \ | \ Y \subseteq V\}.$$

\begin{example}\label{p4}
{\rm    Let $G=P_4$ and let $X=\{1,2,3\} \subseteq [4]$. Then $$N[X]=N[1] \cap N[2] \cap N[3]=\{1,2\} \cap \{1,2,3\} \cap \{2,3,4\}=\{2\}$$ and $$\mathcal{N}=\{\varnothing,\{2\},\{3\},\{1,2\},\{2,3\}, \{3,4\}, \{1,2,3\},\{2,3,4\},[4]\}.$$}
\end{example}

\begin{defn}[\cite{1020}]
    {\rm Let $G=(V,E)$ be a a graph and let $X \subseteq V$. The {\bf neighbourhood closure} of $X$ in $G$ is defined by $$\hat{X}:=N[N[X]].$$

$X$ is called {\bf neighbourhood closed} if $\hat X=X.$ If  $x \in V$, we write $\hat{x}$, instead of $\reallywidehat{\{x\}}$.  
}
\end{defn}

The following proposition is about a series of properties of the common closed neighbourhood and of the neighbourhood closure. 

\begin{prop}
\label{101}
    Let $G=(V,E)$ be a graph and $A,B\subseteq V$. Then the following facts hold:
    \begin{enumerate}[label=$(\arabic*)$, ref=\thetheorem (\arabic*)]
        \item \label{101first} If $A \subseteq B $, then we have $N[A] \supseteq N[B]\supseteq \mathcal{S}$.
        \item \label{101second} If $A \subseteq B $, then  $A\cup \mathcal{S} \subseteq \hat{A} \subseteq \hat{B}$.
        \item \label{101third} $N[\hat{A}]=N[A]$ and $\widehat{(\hat{A})}=\hat{A}$.
        \item \label{101fourth} $N[A \cup B]=N[A] \cap N[B]$ and $\widehat{A\cup B}\supseteq \hat{A}\cup\hat{B}$. 
        \item \label{101fifth} If $x \in V$, then $\hat{x} \subseteq N[x]$.
        \item \label{101sixth} Let $A\neq\varnothing $. Then, for every $v \in A$, we have  $N[A] \subseteq C(v)$. If $N[A]=V$, then $G$ is a connected graph and $A\subseteq \mathcal{S}$.
        \item \label{101seventh} If $A\neq\varnothing $ is such that $N[A] \not = \varnothing$, then, for every $x,y \in N[A]$, we have $d(x,y)\leq 2$.
    \item \label{101ninth}   $B\subseteq \hat A$ if and only if $N[B]\supseteq N[A].$ In particular, $\hat A=\{y\in V:N[y]\supseteq N[A]\}.$ 
  \item \label{101tenth}
  $\hat A=\bigcup_{ N[Y]\supseteq N[A]} Y.$ 

  \end{enumerate}
\end{prop}
\begin{proof}
  
$(1)$-$(4)$ Those are the content of \cite{1020}[Proposition 2$(i)$-$(iv)$].
\begin{enumerate}
\item[$(5)$] Let $x \in V$. Since $\{x\} \subseteq N[x]$, then, by $(1)$, we have $N[x]=N[\{x\}] \supseteq N[N[x]]=\hat{x}.$ 

      \item[$(6)$] Let $\varnothing \not = A \subseteq V$ and let $v \in A$. If $N[A]= \varnothing$, then we trivially have $N[A]\subseteq C(v)$. If $N[A]\not = \varnothing$ then, for every $x \in N[A]$, we have $x \in N[v]$. But $N[v] \subseteq C(v)$ and so $N[A] \subseteq C(v)$.

      Assume next $N[A]=V$. Then, for every $v \in A$, we have $C(v)=V$, that is $G$ is a connected graph. Finally if $x \in V$, then, for every $v \in A$, we have $x \in N[v]$, that is $v$ is a star vertex of $G$.
      \item[$(7)$] Let $\varnothing \not = A \subseteq V$ be such that $N[A] \not = \varnothing$ and let $x,y \in N[A]$. Since $A \not = \varnothing$, there exists $z \in A$ and $x,y \in N[z]$. Then, $d(x,y) \leq 2$.
  \item[$(8)$]
  Let $B\subseteq \hat A$. Then, by $(1)$ and $(3)$, $N[B]\supseteq N[\hat A]=N[A].$  Conversely, assume that $N[B]\supseteq
  N[A].$ Then, by $(1)$, $B\subseteq \hat B\subseteq \hat A.$ Now, the description of $\hat A$ given by $\hat A=\{y\in V:N[y]\supseteq N[A]\},$ follows immediately. 
  
 \item[$(9)$] By $(10)$, we immediately have $\bigcup_{ N[Y]\supseteq N[A]} Y\subseteq \hat A.$ Pick now $x\in \hat A.$ Then, by $(10)$, $N[x]\supseteq N[A]$. Thus $\bigcup_{ N[Y]\supseteq N[A]} Y\supseteq \{x\}.$ 
 As a consequence $\hat A=\bigcup_{ N[Y]\supseteq N[A]} Y.$
 

\end{enumerate}  
\end{proof}

\begin{oss}
{\rm Let $G=(V,E)$ be a graph and $A \subseteq V$. If $A= \varnothing$, then Proposition~\ref{101seventh} does not hold. Indeed, consider $G=P_4$. Then we have $N[A]=V=[4]$ and $d(1,4)=3 \not \leq 2$.}
    
\end{oss}

\begin{cor}
\label{340}
   Let $G=(V,E)$ be a graph. Then the following facts hold:
    \begin{enumerate}[label=$(\arabic*)$, ref=\thetheorem (\arabic*)]
        \item \label{340first} The family $\mathcal{N}$ is closed by intersections and coincides with the set of the neighbourhood closed subsets of $V$. Moreover, $\mathcal{N}$ is a convexity on $V$ if and only if $G$ has no star vertex.
        \item \label{340second} The function that maps $X \in \mathcal{P}(V)$ in $\hat{X} \in \mathcal{P}(V)$ is an extensive, isotone and idempotent  
        closure operator on $V$, that is, a Moore closure operator on $V$. We call such operator, the {\bf hat operator}. 
    \end{enumerate} 
    
\end{cor}
\begin{proof}
    \begin{enumerate}
        \item[$(1)$] The fact that $\mathcal{N}$ is closed by intersection follows from Proposition~\ref{101fourth}.
        
        We now show that if $X \subseteq V$, then we have $X\in \mathcal{N}$ if and only if $\hat{X}=X$. If $X \subseteq V$ is such that $X=\hat{X}$, then we have $X=N[N[X]] \in \mathcal{N}$.
        On the other hand, let $X \in \mathcal{N}$. Then we have $X=N[Y]$ for some $Y \subseteq V$. Thus, by Proposition~\ref{101third}, we have $$\hat{X}=N[N[N[Y]]]=N[\hat{Y}] =N[Y]=X.$$ 

Since $\mathcal{N}$ is closed by intersections  and $V=N[\varnothing]$,  we have that $\mathcal{N}$ is a convexity on $V$ if and only if $\varnothing \in \mathcal{N}$. We have already noted that $N[V]$ is the set of star vertices in $G$. So if $G$ has no star vertices, then $\varnothing=N[V] \in \mathcal{N}$. On the other hand if $N[V] \not = \varnothing$, then, for every $A \subseteq V$,  we have $N[A] \supseteq N[V] \not = \varnothing$ and so $\varnothing \not \in \mathcal{N}$. Thus $\mathcal{N}$ is a convexity on $V$ if and only if $G$ has no star vertex.

\item[$(2)$] This is immediate from Proposition~\ref{101first}-\ref{101third}.
\end{enumerate}\end{proof}

We now naturally extend $\mathcal{N}$ in order to obtain a convexity, whatever the graph is.
\begin{defn}{\rm\label{def:neighborconvexity}
    Let $G=(V,E)$ be a graph. The {\textbf{neighbourhood convexity}} on $V$ is defined by $$n:=\mathcal{N} \cup \{\varnothing\}.$$}
\end{defn}


    Note that, for every $Y\subseteq V,$ the set $K\coloneq N[Y]$ is $n$-convex. 
    Moreover if $\varnothing \neq K\subseteq V$ is an  $n$-convex set, then there exists $Y\subseteq V$ such that $K=N[Y].$ If further $K\neq V,$ then such a $Y$ is necessarily non-empty. In particular, for every $X\subseteq V$, $\hat X=N[N[X]]$ is $n$-convex.

    
 We exhibit a first interesting set of $n$-convex sets and establish some basic properties of $n$-convex sets.
\begin{lemma}\label{max-clique}
Let $G=(V,E)$ be any graph. If $X\subseteq V$ is a maximal clique, then $X=N[X]$ is $n$-convex. 
\end{lemma}
\begin{proof}
 Since $X$ induces a complete subgraph, we have $X\subseteq N[x]$ for all $x\in X$. Thus $X\subseteq \bigcap_{x\in X}N[x]=N[X]$. Reciprocally, let $y\in N[X]$ which means that $y\in N[x]$ for all $x\in X$. Since $X$ is maximal, $y$ must belong to $X$ and hence $N[X]=X$. In particular, $X\in \mathcal{N}\subseteq n.$ 
\end{proof}

\begin{prop}
\label{102}
Let $G=(V,E)$ be a graph and $v\in V$. Then the following facts hold:
\begin{enumerate}[label=$(\arabic*)$, ref=\thetheorem (\arabic*)]
    \item \label{102first} If $\varnothing \not = K \subsetneq V$ is an $n$-convex set, then for every $x,y \in K$, $d(x,y)\leq 2$. In particular, $K \subseteq C(x)$ for all $x \in K$.
    \item \label{102second} Every non-empty $n$-convex set contains $\mathcal{S}$. 
    \item \label{102third} A non-empty subset $X$ of $V$ is $n$-convex if and only if, for every $Y\subseteq V$, $N[Y]\supseteq N[X]$ implies $Y\subseteq X.$
\end{enumerate}
\end{prop}
\begin{proof} 
$(1)$ Let $\varnothing \neq K \subsetneq V$ be an $n$-convex set. Then there exists $\varnothing \neq A\subseteq V$ such that $K=N[A].$ Then, by Proposition~\ref{101seventh}, for every $x, y\in K$, we have $d(x,y)\leq 2.$ In particular,  $y\in C(x).$ 
\smallskip

$(2)$ Let $\varnothing \neq K \in n$. Then, by Corollary~\ref{340first} and Proposition~\ref{101second}, $K=\hat K\supseteq K\cup \mathcal{S}$.
\smallskip

 $(3)$  Assume that $\varnothing \neq X\subseteq V$ is $n$-convex. Then, by Corollary~\ref{340first}, we have $\hat X=X.$
 Let $Y\subseteq V$ be such that  $N[Y]\supseteq N[X].$ Then by Proposition~\ref{101tenth}, we have $X= \hat X \supseteq Y.$
Conversely, assume that, for every $Y\subseteq V$, $N[Y]\supseteq N[X]$ implies $Y\subseteq X.$  Then, by Proposition~\ref{101tenth},
we have
$\hat X=\bigcup_{ N[Y]\supseteq N[X]} Y\subseteq X$ and hence $X=\hat X$. Thus, by Corollary~\ref{340first}, we deduce that $X$ is $n$-convex.
\end{proof}

Prior to discussing convex geometries in the next section, it is necessary to clarify the notions of convex hull and extreme vertices in the context of $n$-convexity.

\begin{prop}\label{hathull} Let $G=(V,E)$ be a graph and $X\subseteq V$. Then
\begin{enumerate}[label=$(\arabic*)$, ref=\thetheorem (\arabic*)]
    \item[$(1)$] \label{hathull first}
   $
    [X]_n=\left\{
\begin{array}{rl}
\varnothing \ \mbox{if} \ X=\varnothing,  \\
\hat X  \  \mbox{if} \ X\neq \varnothing.
\end{array}
\right.$
\item[$(2)$] \label{hathull second} If $\hat X=X,$ then $X$ is $n$-convex.
\item[$(3)$] \label{hathull third} If $X\neq \varnothing$ is $n$-convex, then $\hat X=X.$ 
\item[$(4)$] \label{hathull fourth} For every $n$-convex set $K\subseteq V,$ with $|K|\geq 2$, we have  $\mathrm{ex}_n(K)\cap \mathcal{S}=\varnothing.$ 
\end{enumerate}    
    \end{prop}
 \begin{proof}

$(1)$ In any convexity, the convex hull of the empty set is empty and thus $[\varnothing]_n=\varnothing.$ Assume next that $\varnothing\neq  X\subseteq V$. We show that $[X]_n=\hat X.$ We know that $\hat X\in n$ and, by Proposition~\ref{101second}, we have $\hat X\supseteq X$.
Thus $\displaystyle{[X]_n=\bigcap_{K\in n, K\supseteq X}K\subseteq \hat X.}$ Now, $[X]_n\in n$ and $[X]_n\supseteq X\supsetneq \varnothing.$ Thus $[X]_n\in \mathcal{N}$ and hence, by Corollary~\ref{340first} and the isotonicity of the hat operator established in Proposition~\ref{101second}, we deduce $[X]_n=\reallywidehat{[X]_n}\supseteq \hat X $. 
As a consequence, $[X]_n=\hat X$ as desired. 
\smallskip

$(2)$ By Corollary~\ref{340first}, if $X=\hat X,$ we  have $X\in \mathcal{N}\subseteq n.$
 \smallskip
 
  $(3)$ If $X\neq \varnothing$ and is $n$-convex, then $X\in \mathcal{N}$ and thus 
  by Corollary~\ref{340first},  $\hat X=X.$
  \smallskip
  
  $(4)$ Assume, by contradiction, that there exist an $n$-convex set $K\subseteq V$ of size at least $2$ and  $x\in \mathrm{ex}_n(K)\cap \mathcal{S}.$ Then $\varnothing\neq K- \{x\}\not\supseteq\mathcal{S}$. Thus, by Proposition~\ref{102second}, $K- \{x\}\notin n,$ against $x\in \mathrm{ex}_n(K).$
\end{proof} 

\section{The neighbourhood convex geometry}
\label{sec:convexgeometry}

In this section we start our study of the $n$-convex geometry, giving information on true twins, degree of vertices, diameter, star vertices and connectedness.

\subsection{True twins and $n$-convexity}

We first describe an $n$-convex set by the true twins equivalence classes of its vertices.
This easy result is a preparatory one for obtaining  the absence of pairs of true twins in an $n$-convex geometry (Theorem~\ref{noclosedtwins}). 
\begin{lemma}\label{lemma twin}
 Let $G=(V,E)$ be a graph. Then, the following facts hold:
 \begin{enumerate}[label=$(\roman*)$, ref=\thetheorem $(\roman*)$]
     \item \label{lemma twin first} Let $\varnothing\neq X\subseteq [x]_N$. Then  $N[X]=N[x]$ and $\hat X=\hat x.$
     \item \label{lemma twin second} Let $ K\in n$ and $x\in K$. Then $K\supseteq [x]_N.$
     In particular, every non-empty $n$-convex set is union of true twin classes.
     \item \label{lemma twin third} Let $X\subseteq V$. Then 
 $$\hat X\supseteq \bigcup_{x\in X}\reallywidehat{[x]_N}\supseteq \bigcup_{x\in X}[x]. $$
 \end{enumerate}
\end{lemma}
\begin{proof} $(i)$ Since $X\neq \varnothing,$ we have $N[X]=\cap_{z\in X} N[z]=\cap_{z\in X}N[x]=N[x].$ It follows that $\hat X=N[N[X]]=N[N[x]]=\hat x.$
\smallskip

$(ii)$ By $\{x\}\subseteq K,$ using $(i)$, we deduce 
$K=\hat K\supseteq \hat x=\reallywidehat{[x]_N}\supseteq [x]_N.$
\smallskip

$(iii)$ Let $x\in X$. By $\{x\}\subseteq X,$ using $(i)$ on $[x]_N\neq \varnothing$, we deduce $\hat X\supseteq \hat x=\reallywidehat{[x]_N}.$ The last inclusion is obvious.
\end{proof}
\begin{theorem}\label{noclosedtwins}
Let $G$ be an $n$-convex geometry. Then $G$ is true twins free. 
\end{theorem}

\begin{proof} 
Let $G=(V,E)$ and assume, by contradiction, that there exists $x\in V$ such that $|[x]_N|\geq 2.$
Our claim is that $\reallywidehat{[x]_N}$ is a non-empty $n$-convex set without extreme points and so, we reach a contradiction.    Obviously, $\reallywidehat{[x]_N}$ is non-empty and $n$-convex. We show that, for every $y\in \reallywidehat{[x]_N},$ the set 
 $S\coloneq\reallywidehat{[x]_N}\setminus \{y\}$ is not $n$-convex. By $|[x]_N|\geq 2$, we deduce that $\varnothing \neq [x]_N\setminus \{y\}\subseteq S$ and thus $S\neq \varnothing$. Therefore, by Proposition~\ref{hathull third}, it is enough to show that $S\neq \reallywidehat{S}$. 
Since $[x]_N\setminus\{y\}\neq \varnothing$, applying Lemma~\ref{lemma twin first},
we indeed have $$\hat S=\reallywidehat{\left(\reallywidehat{[x]_N}\setminus \{y\}\right)}\supseteq \reallywidehat{[x]_N\setminus\{y\}}=\reallywidehat{[x]_N}\neq S.$$ 
Finally, assume, by contradiction, that  there exist  $s\neq  s'\in \mathcal{S}$. Then   $s, s'$ are a  pair of true twins, because $N[s]=V=N[s']$, against the fact that no pairs of true twins exists.
\end{proof}
Note that it is instead possible to have a pair of false twins in an $n$-convex geometry. For instance, this happens in the $n$-convex geometry $C_4$ (see Corollary~\ref{regular2}).

\subsection{Graphs which are $n$-convex geometries}
As we mentioned in the introduction, one of the main differences between $n$-convexity and the classical graph convexities built through paths, is that there are comparatively few graphs which are an $n$- convex geometry. In this section, we will be more precise about this fact.

We begin to explore which graphs are $n$-convex geometries. The next theorem tells us that there is only one disconnected graph, up to isomorphisms, so this happens: the disconnected graph $2K_1$ on two vertices. This is something that, does not happen in convex geometries $\mathscr{C}$ defined by paths. Indeed, those geometries are typically hereditary, and thus  a graph is a $\mathscr{C}$-convex geometry if and only if all its connected components  are $\mathscr{C}$-convex geometries.
\begin{theorem}
\label{103}
    A disconnected graph $G$ is an $n$-convex geometry if and only if $G=2K_1.$
\end{theorem}
\begin{proof} If $G=2K_1$, it is immediately checked that $n$ is the free convexity and thus $G$ is an $n$-convex geometry.
Conversely, suppose that $G=(V,E)$ is disconnected and an $n$-convex geometry. Then $|V| \geq 2$ and we want to show that, indeed, $|V|=2$.  Assume, to the contrary, that $|V|\geq 3$. Let $v \in \operatorname{ex}_n(V)$. Then we have $V-\{v\} \in n$ and $\varnothing \not =V-\{v\}\subsetneq V$. Let $u \in V-\{v\}$. Then, by Proposition~\ref{102first}, we have $V-\{v\} \subseteq C(u)$. Since $G$ is disconnected, we necessarily have $V-\{v\}=C(u)$ and $v$ is an isolated vertex of $G$. So, $\{v\}=N[x]\in n$. Since $|V| \geq 2$ and $(G,n)$ is a convex geometry, then, by Theorem~\ref{62second}, there exists $w \in V-\{v\}$ such that $\{v,w\}$ is an $n$-convex set. 
But $v$ is isolated and so there does not exist a connected component of $G$ which contains both $v$ and $w$. This contradicts Proposition~\ref{102first}.
\end{proof}
In the next proposition we see that the only complete graph  which is an $n$-convex geometry, is the graph with one vertex.
\begin{prop}\label{104} The only complete graph which is an $n$-convex geometry is $K_1.$
\end{prop}
\begin{proof} Let $G=K_n=(V,E)$ for some $n\geq 1.$ Then, for every $A \subseteq V$, we have $N[A]=V$ and thus $n=\{\varnothing,V\}$ is the coarse convexity. By Theorem~\ref{62second}, there exists $x \in V$ such that $\{x\} \in n$. Since $n=\{\varnothing,V\}$, then we necessarily have $V=\{x\}$, that is $G =K_1$.
\end{proof}
Note that, by the above theorem, $P_2\cong K_2$ and  $C_3\cong K_3$ are not $n$-convex geometries. Theorem\ref{103} also implies that a graph is almost never an hereditary $n$-convex geometry.


\begin{theorem}\label{nohered}
    Let $G$ be a graph. Then  $(G,n)$ is an hereditary convex geometry if and only if $G\in\{K_1, 2K_1\}.$
\end{theorem}
\begin{proof}

If $G \in\{K_1, 2K_1\}$, then  $G$ is an $n$-convex geometry by Theorem~\ref{103}. The fact that it is an hereditary $n$-convex geometry immediately follows.
 Suppose next that $(G,n)$ is an hereditary convex geometry. We claim that $E = \varnothing$. Assume the contrary. Then there exist $u \not = v \in V$ such that $\{u,v\} \in E$. Then the subgraph $H:=(\{u,v\},\{\{u,v\}\})$ is induced in $G$. As a consequence, $(H,n_H)$ is a convex geometry. But this contradicts Proposition~\ref{104}, since $H \cong K_2$. 
Now, if $G$ is connected, then we necessarily have $G = K_1$. If $G$ is disconnected, since $(G,n)$ is a convex geometry, by Theorem~\ref{103}, we deduce $G=2K_1$. 
\end{proof}
The following theorem  gives strong limitations for the graphs which are $n$-convex geometries.
As usual, for a graph $G$, we denote by $\delta(G)$ the minimum degree of a vertex of $G.$
\begin{theorem}
\label{333}
 Let $G=(V,E)$ be a graph which is an $n$-convex geometry. Then the following facts hold:
 \begin{enumerate}[label=$(\arabic*)$, ref=\thetheorem (\arabic*)]
     \item \label{333first} $|V| \leq 2\, \delta(G) +2$. 
     \item \label{333second} If $|V| \geq 3$, then $G$ is connected and $2 \leq \mbox{\upshape{diam}}(G) \leq 3$.
 \end{enumerate}    
\end{theorem}
\begin{proof}
\begin{enumerate}
\item[$(1)$] Let $k:=\delta(G)$ and 
 $v \in V$ be such that $k=d(v)$. We want to show $|V| \leq 2k+2.$ Assume by contradiction that $|V|\geq 2k+3$. We prove the following claim.\\
  \noindent\emph{Claim: For every $i \in [k+1]$, there exist $v_1,...,v_i\in V$ distinct such that $N[v] \cupdot \{v_1,...,v_j\}$ is $n$-convex for all $j\leq i$.}\\ 
  We argue by finite induction on the set $[k+1]$.
  Since $N[v]$ is an $n$-convex set and $|N[v]|=k+1 < 2k+3 \leq |V|$ then, by Theorem~\ref{62third}, there exists $v_1 \in V-N[v]$ such that $N[v] \cupdot \{v_1\}$ is an $n$-convex set. Thus the statement in the claim holds for $i=1.$ If $k=0$ the claim is proved. 
  Suppose now that $k\geq 1$ and, for some $i \in [k]$, the statement in the claim holds. Thus, we have found $v_1,...,v_i \in V- N[v]$ pairwise distinct such that $N[v] \cupdot \{v_1,...,v_j\}$ is an $n$-convex set for all $j\leq i$. Since $|N[v] \cupdot \{v_1,...,v_i\}|=k+1+i \leq 2k+1 < 2k+3 \leq |V|$ then, by Theorem~\ref{62third}, there exists $v_{i+1} \in V-(N[v] \cupdot \{v_1,...,v_i\})$ such that $ N[v] \cupdot \{v_1,...,v_i,v_{i+1}\}$ is an $n$-convex set. 
  Thus $ N[v] \cupdot \{v_1,...,v_j\}$ is an $n$-convex set for all $j\leq i+1$.
  This shows that the statement in the claim holds also for $i+1.$
  Hence the claim is proved.

We now exploit the claim for $i=k+1$. It assures that we have $v_1,...,v_{k+1} \in V$ pairwise distinct such that, for every $i \in [k+1]$, $N[v] \cupdot \{v_1,...,v_i\}\in \mathcal{N}$ and thus also $$|N[v] \cupdot \{v_1,...,v_i\}|=k+1+i \leq 2k+2 < 2k+3 \leq |V|.$$ As a consequence, for every $i \in [k+1]$, we have $ N_G[v] \cupdot \{v_1,...,v_i\} \subsetneq V$ and therefore there exists $\varnothing \not = V_i \subseteq V$ such that $N[V_i]=N[v] \cupdot \{v_1,...,v_i\}$. 
Since, for every $i \in [k+1]$, we have $N[v] \subseteq N[V_i]$, using Proposition~\ref{101first},~\ref{101second} and~\ref{101fifth}, we deduce that  $$V_i \subseteq \reallywidehat{V_i} \subseteq \hat{v} \subseteq N[v].$$ 
We claim that, for every $i \in [k+1]$, we have 
\begin{equation}\label{inclusion}
 V_i \subseteq N(v).   
\end{equation}
Indeed, assume by contradiction that 
there exists $i \in [k+1]$ such that $v \in V_i$. Then we have $N[V_i] \subseteq N[v]$, against $v_1 \in N[V_i]-N[v]$. 

Now, for every $i \in [k]$, we have $v_{i+1} \not \in N[V_i]$ and, so there exists $w_i \in V_i$ such that $v_{i+1} \not \in N[w_i]$. But $v_{i+1} \in N[V_j]$ for all $j \in [i+1,k+1]$ and so 
\begin{equation}\label{grande}
 w_i \not \in \bigcup_{j \in [i+1,k+1]}V_j, \quad \hbox{for all}\quad i \in [k].  
\end{equation}
In particular, we have 
\begin{equation}\label{ultima}
\{w_1,...,w_k\} \cap V_{k+1}=\varnothing.
\end{equation}
Now, by \eqref{inclusion}, we have $\{w_1,...,w_k\} \subseteq N(v)$ and, by \eqref{grande}, $w_1,...,w_k$ are pairwise distinct. Since $|N(v)|=k$, then we necessarily have $N(v)=\{w_1,...,w_k\}$. As a consequence, we have $V_{k+1} \subseteq \{w_1,...,w_k\}$ and so, by \eqref{ultima}, we obtain $V_{k+1}=\varnothing$, a contradiction.
\smallskip

\item[$(2)$] Since $|V|\geq 3$, by Theorem~\ref{103}, $G$ is connected. Moreover, by Theorem~\ref{104}, $G$ is not a complete graph and thus $k\coloneq \mbox{diam}(G) \geq 2$. Suppose, by absurd, that $\mbox{diam}(G) \geq 4$. Let $v,w \in V$ and $\gamma:v=v_0 v_1 ... v_k=w$ be a $v-w$ geodesic in $G$ such that $d(v,w)=k \geq 4$. By Lemma~\ref{202} with $K:=\varnothing$ and $X\coloneq V(\gamma)$, there exists $i \in [k]_0$ and an $n$-copoint $D$ attached at $v_i$ such that $V(\gamma)-\{v_i\} \subseteq D$.

Now if $i=0$, then we have $v_1,w \in D$ and $d(v_1,w)=k-1\geq 3$. If $i \in [k-1]$, then we have $v,w \in D$ and $d(v,w)=k \geq 4$. If $i=k$, then we have $v,v_3 \in D$ and $d(v,v_3)=3$. By Proposition~\ref{102first}, we obtain that in every case $D$ is not an $n$-convex set, a contradiction.
\end{enumerate}
\end{proof}

\begin{oss}
\label{341}{\rm
    Observe that the inequalities $(1)$ and $(2)$ in Theorem~\ref{333} are both sharp. 
    Indeed, $2K_1$ is an $n$-convex geometry and the degree of every vertex is $0$. So in $(1)$ equality holds.
    
    Consider next the graphs $P_3$ and $P_4$. They are connected graphs with at least $2$ vertices and  are $n$-convex geometries. Indeed, it is a routine exercise to check that $$n_{P_3}=\{\varnothing,\{2\},\{1,2\},\{2,3\}, [3]\}$$ and $n_{P_4}$ has been computed in Example~\ref{p4}.
    
    Now, using Theorem~\ref{62third}, it is easily checked that 
    in both cases we get an $n$-convex geometry.
   Since $\mbox{diam}(P_3)=2$ and $\mbox{diam}(P_4)=3$, we have that $P_3$ and $P_4$ realize one  of the equalities in $(2)$.}
\end{oss}

We know that in a convex geometry there exist convex sets of any size.   We can immediately characterize those of size one in an $n$-convex geometry with stars.
Later, in Corollary~\ref{Dintersect}, we will characterize the convex sets of size one in an $n$-convex geometry with no stars.

\begin{lemma}\label{conv-star}
Let $G=(V,E)$ be an $n$-convex geometry,
Then $|\mathcal{S}|\leq 1$. Moreover, if $|\mathcal{S}|=1$, the only $n$-convex set of $G$ of size one is $\mathcal{S}.$ 
\end{lemma}
\begin{proof}
First note that  $\mathcal{S}=N[V]\in n.$  Assume, by contradiction, that $|\mathcal{S}|\geq 2.$
Since $G$ is an $n$-convex geometry, there exists a vertex $s\in \mathcal{S}$ which is extreme for $\mathcal{S}$. Thus $\varnothing\neq\mathcal{S}-\{s\}\in n,$ against Proposition~\ref{102second}. 

Assume now that $|\mathcal{S}|\leq 1$ and let $K\in n$, with $|K|=1.$ Then, by Proposition~\ref{102second}, we have $K\supseteq \mathcal{S}$ and thus $K=\mathcal{S}.$ Since we know that $\mathcal{S}\in n$, this shows that the only $n$-convex set of size $1$ is $\mathcal{S}$.
\end{proof}

Another restriction on the family of $n$-convex geometries is the following result.

\begin{prop}\label{famiglie}
Let $G=(V,E)$ be an $n$-convex geometry with $|V|=k\geq 2$. Then the following facts hold:
    \begin{enumerate}[label=$(\roman*)$, ref=\thetheorem $(\roman*)$]
        \item \label{famiglie first} $\frac{k(k-2)}{4}\leq |E|\leq \frac{(k-1)^2}{2}$.
        \item  \label{famiglie second} If $k\geq 5$, then $C_4$ is a subgraph of $G$. In particular, $G$ cannot be a forest.
\end{enumerate}    
\end{prop}

\begin{proof}  

$(i)$ By the Handshaking Lemma, we have $2|E|=\sum_{x\in V}d(x)$. Since, for every $x\in V$, by Theorem~\ref{333}, we have $d(x)\geq \delta(G)\geq \frac{k-2}{2}$, we obtain $|E|\geq \frac{k(k-2)}{4}.$

Observe now that we have at most one vertex $s$ with $d(s)=k-1$ and, for the remaining $x\in V$, we have $d(x)\leq k-2.$ Thus we get 
$$2|E|\leq d(s)+\sum_{x\in V\setminus\{s\}}d(x)=(k-1)+(k-1)(k-2)=(k-1)^2.$$
\smallskip

$(ii)$ We recall that, by Reiman's Theorem, if $$|E|>\frac{k}{4}(1+\sqrt{4k-3}),$$
then $G$ admits $C_4$ as subgraph. 
Now, by Theorem~\ref{333}, we know $|E|\geq \frac{k(k-2)}{4}$ and an easy computation assures that $ \frac{k(k-2)}{4}>\frac{k}{4}(1+\sqrt{4k-3})$ for $k\geq 9$. The small cases $5\leq k\leq 8$ are checked by computer.
\end{proof}

\subsection{Quasi-stars }\label{sec: Quasi-stars }
We introduce a special type of vertices around which we develop a large part of our paper. 

\begin{defn}\label{D}{\rm Let $G=(V,E)$ be a graph. We set
$$\mathcal{D}\coloneq\{x\in V:d(x)=|V|-2\}.$$ 
The vertices in $\mathcal{D}$ are called {\bf quasi-stars}. Note that if $|V|=1$ the set $\mathcal{D}$ is necessarily empty; if $|V|=2$, it consists of isolated vertices.\\
We also define a function $v:\mathcal{D}\rightarrow V$. If $\mathcal{D}=\varnothing$, then $v$ is the so-called empty function and it has an empty image. If $\mathcal{D}\neq \varnothing$, then $v$ is defined, for every $x\in \mathcal{D}$, by 
$$\{v(x)\}\coloneq V- N[x].$$
We denote by $\mathcal{T}$ the subset of $V$ given by the image of $\mathcal{D}$ by $v$. 
}
\end{defn}


\begin{theorem}\label{stars-extreme} Let $G=(V,E)$ be a graph with at least two vertices. Then $\mathcal{T}=\operatorname{ex}(V)$ and $|\mathcal{D}|\geq |\operatorname{ex}(V)|$. Moreover,
$V=N[\mathcal{D}]\cupdot \operatorname{ex}(V).$
Under the assumption that $G$ is true twins free,  the map $v$ considered from  $\mathcal{D}$ to $\operatorname{ex}(V)$ is a bijection. 
Finally, if $G$ is an $n$-convex geometry, then  $|\operatorname{ex}(V)|=|\mathcal{D}|\geq 2.$
\end{theorem}
\begin{proof}
We claim that, 
$$\  (\dag)\  \hbox{for every}\  x\in \operatorname{ex}(V),\  \hbox{there exists}\  y\in \mathcal{D}\  \hbox{ such that}\  \{x\}=V- N[y].$$
Indeed, let $x\in \operatorname{ex}(V)$. Then $V-\{x\}\in n$ and $V-\{x\}\notin \{V,\varnothing\}.$ Thus there exists $Y\subseteq V$ such that $V-\{x\}=N[Y]=N[Y- \mathcal{S}]$ and, necessarily, $Y-\mathcal{S}\neq\varnothing.$ 
Let $y\in Y-\mathcal{S}$. Then $V-\{x\}=N[Y- \mathcal{S}]\subseteq N[y],$ which implies $d(y)\geq n-2.$ Now, we cannot have $d(y)=n-1$, because $y\notin \mathcal{S}.$ Thus $d(y)= n-2$,  $y\in \mathcal{D}$  and  $V-\{x\}=N[y],$ that is $\{x\}=V- N[y].$

Now, by Definition~\ref{D}, $(\dag)$ translates into $x=v(y)$. Thus $(\dag)$ assures that $\operatorname{ex}(V)\subseteq \mathcal{T}.$ 
It remains to show that $\mathcal{T}\subseteq \operatorname{ex}(V).$
Pick $x\in \mathcal{T}$. Then there exists $y\in \mathcal{D}$ such that $x=v(y)$, which implies $\{x\}=V- N[y]$. It follows that $V- \{x\}=N[y]\in n$ and thus $x\in \operatorname{ex}(V).$

Observing that the map $v:\mathcal{D}\rightarrow V$ has image $\mathcal{T}=\operatorname{ex}(V),$ it is obvious that $|\mathcal{D}|\geq |\operatorname{ex}(V)|$.\\
Finally, observe that
\begin{equation}\label{extremes}
\operatorname{ex}(V)=\bigcup_{x\in \operatorname{ex}(V)}\{x\}=\bigcup_{y\in \mathcal{D}}(V-N[y])=V-\bigcap_{y\in \mathcal{D}}N[y]=V-N[\mathcal{D}].
\end{equation}
As a consequence, we have $V=N[\mathcal{D}]\cupdot \operatorname{ex}(V).$
 Note that all the above reasoning are possible also in the case $\mathcal{D}=\varnothing$. In particular, when $\mathcal{D}=\varnothing$ we have $\operatorname{ex}(V)=\varnothing$ and the splitting of $V$ appears as the trivial one $V=V\cupdot \varnothing.$

Assume next that $G$ is true twins free.
The map $v$ thought from $\mathcal{D}$ to $\operatorname{ex}(V)$ is surely surjective. So we only need to check that it is injective. Let $x,y\in \mathcal{D} $ be such that $v(x)=v(y).$ Then $V- N[x]=V- N[y]$, that is, $N[x]=N[y]$. Thus $x$ and $y$ are closed twins, which implies $x=y.$ 

Finally, assume that $G$ is an $n$-convex geometry. Then, by Theorem~\ref{noclosedtwins}, $G$ is true twins free. Thus, by what shown before, we have $|\mathcal{D}|=|\operatorname{ex}(V)|\geq 1,$ because in a convex geometry the set of extreme vertices is never empty.
Assume, by contradiction, that $\operatorname{ex}(V)=\{x\}$, for some $x\in V.$ Then, passing to the $n$-convex hulls and using Proposition~\ref{hathull first},  we have $V=[\{x\}]_n=\hat x.$ Using Proposition~\ref{101third}, we then deduce 
$$\mathcal{S}=N[V]=N[\hat x]=N[x].$$
 Thus $x$ is a star vertex. Since $|V|\geq 2$, this goes against Proposition~\ref{hathull fourth}. 
Thus $|\operatorname{ex}(V)|\geq 2.$
\end{proof}

\begin{cor}\label{regular2}
 Let $k\geq 2$ be even and $G$ be the $(k-2)$-regular graph with $k\geq 2$ vertices. Then $n$ is the free convexity. In particular, $G$ is an $n$-convex geometry. \end{cor}
 \begin{proof}
By definition of $\mathcal{D}$, we have that $\mathcal{D}=V$ and, since $G$ has no star, we have $\varnothing=N[V]=N[\mathcal{D}]$. By Theorem~\ref{stars-extreme}, we have $V=\operatorname{ex}(V).$
 Thus, for every $x\in V,$ we have $V-\{x\}\in n.$  Let $X\subseteq V$. If $X=V$, then obviously $X\in n.$ 
  Assume next that $X\neq V$. Then $V-X=\{y_1, \dots, y_r\}$ for suitable  $y_i\in V, i\in [r], r\geq 1.$
  Since $V-\{y_i\}\in n$ for all $i\in [r],$ we deduce that
  $X=V-\{y_1, \dots, y_r\}=\bigcap_{i\in [r]}V-\{y_i\}\in n.$ Thus $n$ is the free convexity.
\end{proof}

We emphasize that, by Theorem~\ref{stars-extreme}, we have $N[\mathcal{D}]\cap \operatorname{ex}(V)=\varnothing.$
The property $\mathcal{D}\cap \operatorname{ex}(V)=\varnothing$ does not necessarily hold, even assuming that the graph is an $n$-convex geometry. For instance, let $C_4=(V,E)$ and note that, by Corollary~\ref{regular2},  $C_4$ is an $n$-convex geometry. We have $\mathcal{D}(C_4)=V=\operatorname{ex}(V).$ Note also that the property $N[\mathcal{D}] \supseteq \mathcal{D}$ does not necessarily holds. Indeed, $N[\mathcal{D}(C_4)]=\varnothing\not \supseteq \mathcal{D}(C_4).$
\begin{cor}\label{notcomp}
 The class of graphs which are $n$-convex geometries is not self-complementary.   
\end{cor}
\begin{proof} It follows from  
Theorem~\ref{103} and Corollary~\ref{regular2}, because 
 $C_4$ is an $n$-convex geometry while, by Theorem~\ref{103}, its complement $2K_2$ is not.   
\end{proof}

We emphasize that both cases $\mathcal{S}=\varnothing$ and $\mathcal{S}=\{s\}$ in Theorem~\ref{stars-extreme} can arise. Indeed $C_4$ and the graph obtained by $C_4$  adding a star are both $n$-convex geometries. In Section~\ref{construction}, we will see how the idea of adding a star  to an $n$-convex geometry without stars is generally fruitful.
Before, we need to establish a precious link between our neighbourhood convexity and a combinatorial device given by preorders and partial orders, naturally associated with neighbourhood convexity (see Definition~\ref{pre}). This link will shed fresh new light on the set $\mathcal{D}$ of quasi-stars in an $n$-convex geometry (see Proposition~\ref{Dposet} and Corollary~\ref{Dintersect}).

\section{The neighbourhood preorder of a graph}
\label{sec:preorder}
 In this section, we establish a link between neighbourhood convexity and lattice theory by associating our convexity with a preorder (Definition \ref{pre}).
\subsection{Preorders and convexity}
We recall some definitions and elementary properties. Those are surely well-known for posets but, since we need them in the more general framework of preorders, we give them for clarity.
Let $V$ be a set. A binary relation in $V$, that is a subset of $V^2$, 
is called a {\bf preorder} if it is reflexive and transitive; and a {\bf partial order} if it is an antisymmetric preorder. As common, we denote a preorder by $\leq$ and, when $(x,y)\in \leq$, we prefer to write 
 $x\leq y$. Two elements $x,y\in V$ such that $x\not\leq y$ and  $y\not\leq x$ are called {\bf incomparable}; otherwise they are called {\bf comparable}. Recall that in a preorder can exist $x\neq y\in V$ such that both $x\leq y$ and $y\leq x$ hold.
 When $\leq$ is a partial order, we say that $(V,\leq)$ is a {\bf poset}. A poset with no incomparable elements is called a {\bf linear order}.
 
 Let $(V,\leq)$ be a preorder.  If $X\subseteq V$, then
 considering the pairs $(x,y)\in X^2$ such that $x\leq y$ we obtain a relation on $X$, denoted still by $\leq$, which gives to $X$ the structure of 
  a preorder $(X,\leq)$ called induced by $(V,\leq)$ on $X$.

 Let $x\in V$.
 Then $x$ is called {\bf maximal} if, for every $y\in V$, $x\leq y$ implies $y\leq x$; {\bf minimal} if, for every $y\in V$, $y\leq x$ implies $x\leq y$. Given $X\subseteq V$, the set of maximal  (resp. minimal) elements with respect to the preorder induced on $X$ is denoted by $\mathrm{Max}(X)$ (resp. $\mathrm{Min}(X)$). 
  If $X$ is non-empty  and finite, then both $\mathrm{Max}(X)$ and $\mathrm{Min}(X)$ are non-empty. For the rest of poset-related concepts the reader can refer to~\cite{dp02}. 

 Let $U\subseteq V.$ Then $U$ is called an {\bf upset} if $x\in U$ and $x\leq y$ imply $y\in U$.  We denote by $\mathscr{U}(V,\leq)$ the {\bf set of the upsets} of $(V,\leq)$. Given $A\subseteq V$, the {\bf upset with base} $A$ is defined by $$\uparrow{A}\coloneq\{y\in V: x\leq y\textrm{ for some }x\in A\}.$$
 For $x\in V$, instead of $\uparrow{\{x\}}$, we  write $\uparrow{x}$.
  Note  that $\uparrow{A}=\bigcup_{x\in A}
\uparrow{x}$. Similarly one defines {\bf downsets} and related concepts.

\begin{prop}\label{preorder}
Let $(V,\leq)$ be a preorder, with $V$ finite. Then the following facts hold:
\begin{enumerate}[label=$(\roman*)$, ref=\thetheorem $(\roman*)$]
    \item \label{preorder first} The set
$\mathscr{U}(V,\leq)$ 
 is a convexity on $V$ and, for every $X\subseteq V$, $[X]_{\mathscr{U}(V,\leq)}=\uparrow X$.
 \item \label{preorder second}  $V$ is a $\mathscr{U}(V,\leq)$-convex
geometry if and only if
$(V,\leq)$ is a poset.
\end{enumerate}
\end{prop}

\begin{proof} 
$(i)$
This is immediate.
\smallskip

\item$(ii)$ 
 It is well-known that the downsets of  a poset form a convex geometry (\cite[Example II]{Jamison}). Since the upsets of $(V,\leq)$ are just the downsets of its dual poset, we also have that $\mathscr{U}$ is a convex geometry.

Conversely, suppose that $\mathscr{U}\coloneq \mathscr{U}(V,\leq)$ is a convex geometry. Assume, by contradiction, that there exist $x\neq y\in V$ such that 
$x\leq y$ and $y\leq x$. Take the convex set $U=\varnothing$. Then, by $(i),$ we have both \(y \in \uparrow x= [\{x\}]_{\mathscr{U} }\) and 
\(x \in \uparrow y= [\{y\}]_{\mathscr{U} }\) which  violates anti-exchange. Hence $\leq$ is antisymmetric and $(V,\leq)$ is a poset.
\end{proof}

\begin{prop}\label{estremi-preord}
Let $(V,\leq)$ be a preorder and $\mathscr C$ a convexity on $V$ such that $\mathscr C\subseteq \mathscr{U}(V,\leq)$. Then, for every $K\in\mathscr C$,
\[
\operatorname{\operatorname{ex}}_{\mathscr C}(K)\subseteq\operatorname{Min}(K).
\]
\end{prop}

\begin{proof}
Let $x\in\operatorname{\operatorname{ex}}_{\mathscr C}(K)$. By definition,
$K\setminus\{x\}\in\mathscr C$. Hence, by hypothesis,
$K\setminus\{x\}$ is an upset of $(V,\leq)$.
Suppose, by contradiction, that $x\notin\operatorname{Min}(K)$. Then there exists
$y\in K$ such that $y\leq x$ and $x\not\leq y.$
Since $y\neq x$, we have $y\in K\setminus\{x\}$. As
$K\setminus\{x\}$ is an upset and $y\leq x$, it follows that
$x\in K\setminus\{x\}$, a contradiction.
\end{proof}

\subsection{The neighbourhood preorder}\label{sec:pre}
\begin{defn}\label{pre}{\rm 
Let $G=(V,E)$ be a graph.  The {\bf neighbourhood preorder} associated with $G$ is  $P(G)=(V, \leq)$ defined by setting, for  $x,y\in V,\  x\leq y$ if  $N[x]\subseteq N[y]$.}
 \end{defn}
 Note that, since any graph with at least two vertices admits two vertices of the same degree, $P(G)$ is a linear order if and only if $|V|=1.$
 
 If $G$ is true twins free, then the  neighbourhood preorder $P(G)$ is a partial order called  the {\bf neighbourhood poset} associated with $G$. Remarkably, by Theorem~\ref{noclosedtwins}, we know that this happens when $(G,n)$ is a convex geometry.

\begin{lemma}\label{upset1}
Let $G=(V,E)$ be a graph, $P(G)$ be the associated neighbourhood
 preorder, $X\subseteq V$ and $x\in V$. Then:
\begin{enumerate}[label=$(\roman*)$, ref=\thetheorem $(\roman*)$]
\item \label{upset1 first} every non-empty chain in the poset $P(G)$ is a clique in $G$;
    \item \label{upset1 second} $N[\uparrow{X}]=N[X]$ and $\widehat{\uparrow{X}}=\widehat{X}$;
    \item \label{upset1 third} $\uparrow{x}$ is an 
    $n$-convex clique;  $\uparrow{x}=\hat x$; if $x\in \mathrm{Max}(V)$, then $\{x\}\in n$;
   \item \label{upset1 fourth} every $n$-convex set of $G$ is an upset of  $P(G);$ 
   \item \label{upset1 fifth} for every $K\in n$, $\operatorname{\operatorname{ex}}_{n}(K)\subseteq\operatorname{Min}(K).$
\end{enumerate}
\end{lemma}
\begin{proof}

$(i)$ Let $\varnothing\neq C\subseteq V$ be a chain. Pick $x\neq y\in C.$ Since $x$ and $y$ are comparable, up to renaming, we have $x\leq y$ and hence $x\in N[x]\subseteq N[y]$. Thus $x$ and $y$ are adjacent in $G$.
\smallskip

 $(ii)$ We immediately have, 
 \begin{equation}\label{facilissima}N[\uparrow{x}]=\bigcap_{x\leq y}N[y]=\bigcap_{N[x]\subseteq N[y]}N[y]=N[x].
\end{equation} 

Consider now a generic subset $X$ of $V$. Then, using \eqref{facilissima}, we have $$N[\uparrow{X}]=N[\bigcup_{x\in X}\uparrow{x}]=\bigcap_{x\in X}N[\uparrow{x}]=\bigcap_{x\in X}N[x]=N[X].$$
As a consequence, we also immediately have
$$\widehat{\uparrow{X}}=\hat X.$$

$(iii)$ The fact that $\uparrow{x}$ is a clique comes from $(i)$ because $\uparrow{x}$ is a chain in $P(G).$
We now show that $\widehat{\uparrow{x}}=\uparrow{x}$. Clearly, we have $\widehat{\uparrow{x}}\supseteq\uparrow{x}$. We show the other inclusion. Let $y\in\widehat{\uparrow{x}}$. 
 We claim that $N[x]\subseteq N[y]$. 
Indeed,  $z\in N[x]$ implies,  by $(ii)$,    $y\in\cap_{a\in N[x]}N[a]\subseteq N[z]$ and hence $z\in N[y]$. Thus $x\leq y$ and so $y\in\uparrow{x}$.
Now, by $(ii)$, we also have $\hat x=\widehat{\uparrow{x}}=\uparrow{x}.$ Finally, observe that when $x\in \mathrm{Max}(V),$ then $\uparrow{x}=\{x\}.$ 
\smallskip

$(iv)$ Let $X$ be $n$-convex, $x\in X$ and $x\leq y$. We want to show that $y\in X.$ 
By  $x\leq y$, we know that $N[y]\supseteq N[x].$ On the other hand we trivially  have $N[x]\supseteq N[X].$
Thus we obtain $N[y]\supseteq N[X],$ which
by Proposition~\ref{101ninth},  gives $y\in \hat X=X$ as required.
\smallskip

$(v)$ Let $K\in n$. By $(iv)$, $K$ is an upset of $P(G).$ Thus, by Proposition~\ref{estremi-preord}, we deduce $\operatorname{\operatorname{ex}}_{n}(K)\subseteq\operatorname{Min}(K)$.
\end{proof}
We comment the previous result. As a first fact, observe that it is false that in every graph every upset is $n$-convex.
Indeed, consider the graph $G$ having vertex set $V\coloneq [3]_0$ and edge set $E\coloneq \{\{0,i\}:i\in[3]\}$. Note that $G\cong K_{1,3}$ is not an $n$-convex geometry. Clearly $P(G)$ is a poset in which the only inequalities are given by  $i\leq 0 $, with $i\in [3].$ Consider $U\coloneq \{0, 1,2\}$. Then $U=\uparrow{\{1,2\}}$ is an upset. However, $U$ is not $n$-convex  because $\hat{U}=N[N[U]]=N[0]=V\neq U.$


As a second remark, note that several examples show that  the clique $\uparrow{x}$ is not necessarily maximal even under the assumption that the graph is an $n$-convex geometry. See, for instance, $\uparrow{0}$ in Figure~\ref{fig:example017}. 
Since we know that every maximal clique is $n$-convex, Lemma~\ref{upset1} adds new information about some further cliques which turn out to be $n$-convex.
\begin{figure}[tbp]
    \centering \includegraphics[width=0.45\linewidth]{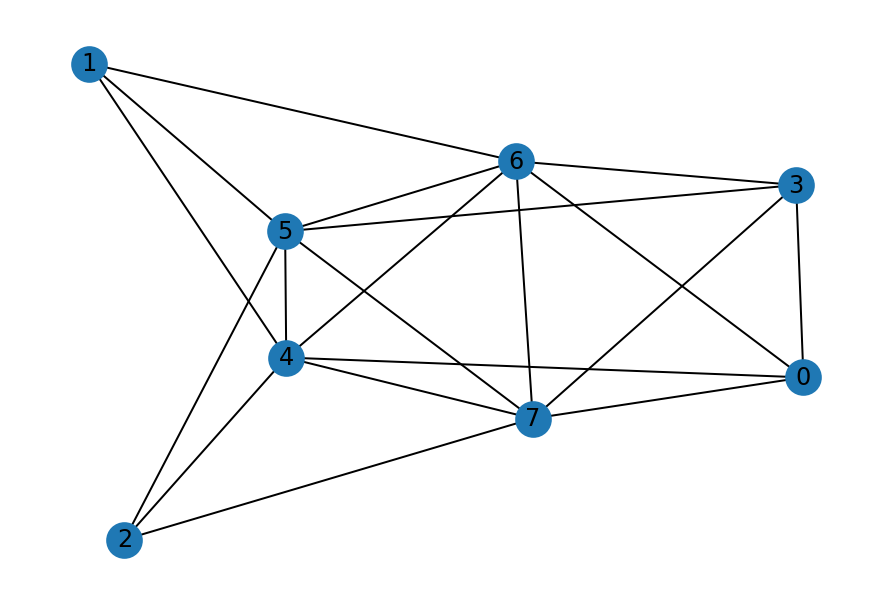} \qquad  \includegraphics[width=0.45\linewidth]{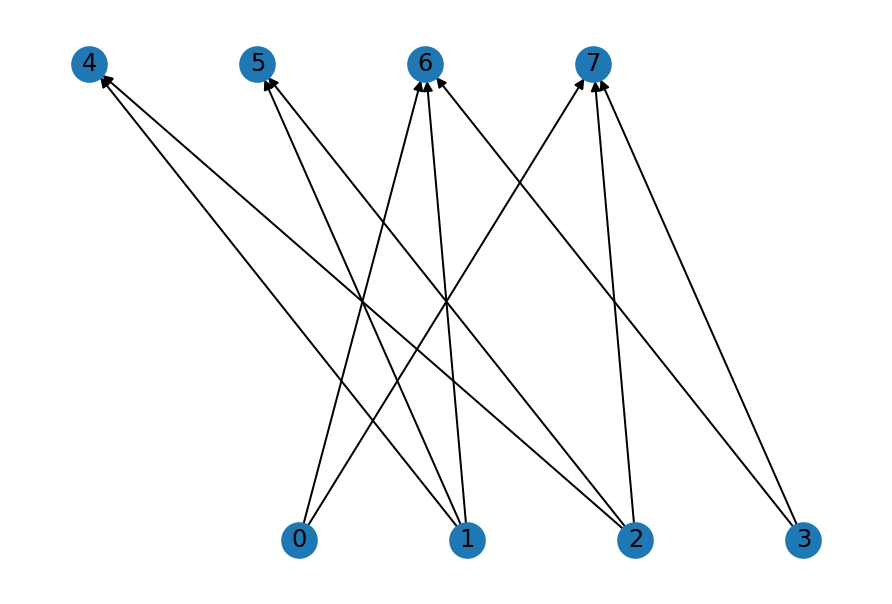}
    \caption{The clique $\uparrow{0}$ is not maximal.}
\label{fig:example017}
\end{figure}

As a third comment, by Lemma~\ref{upset1 third}, we immediately deduce that when $(G,n)$ is a convex geometry, our neighbourhood order relation coincides with the dual of the $\varnothing$-factor relation considered in \cite[p. 250]{Jamison}. This suggests a closer connection with the lattice-theoretic approach to convex geometries, which could be investigated in the future.

Finally, consider for a graph $G=(V, E)$, the {\bf vicinal preorder} $V\! P(G)=(V,\preceq) $ , considered in \cite{Foldes},  and defined  by setting, for every $x,y\in V$,
$x\preceq y$ if $N(x)\subseteq N[y]$.
Surely $x\leq y$ implies $x\preceq y$, and thus $P(G)$ refines $V\!P(G)$. As a consequence $\mathscr{U}(V\!P(G))\subseteq \mathscr{U}(P(G)).$
However, even in the restrictive context of $n$-convex geometries, in Lemma \ref{upset1 fourth} one cannot substitute $P(G)$ with $V\!P(G)$. In other words, it is false that every $n$-convex set is a $V\!P(G)$-upset. Consider, for instance $G=P_4.$ Then $\{1,2\}=N[1]\in n$ but it is not a $V\!P(G)$-upset because $1\in N[1], $ $1\preceq 3$ and $3\notin N[1].$ This easy example also shows that, in general, the neighbourhood preorder and the vicinal preorder do not coincide.
\subsection{Convex geometries and neighbourhood poset}\label{convex_geometries_and_poset}
\begin{prop}\label{Dposet} Let $G=(V,E)$ be an $n$-convex geometry with no star. Then $\mathcal{D}=\mathrm{Max}(V)$.
\end{prop}
\begin{proof} Let $k\coloneq |V|.$ Then $k\geq 2.$ To start with, recall that, since $G$ is an $n$-convex geometry, $P(G)=(V,\leq)$ is a poset.

We show $\mathcal{D}=\mathrm{Max}(V)$. We first show $\mathcal{D}\subseteq \mathrm{Max}(V)$.
Let $x\in \mathcal{D}$ and let $y\in V$ be such that $x\leq y,$ that is, $N[x]\subseteq N[y]$.
We want to show that $y=x.$ Assume, by contradiction, that $y\neq x.$ Then, by antisymmetry, 
$N[x]\subsetneq N[y]$. As a consequence $d(y)\geq d(x)+1=k-1$, which implies the absurd $y\in \mathcal{S}=\varnothing.$ 
We next show that $\mathrm{Max}(V)\subseteq \mathcal{D}$.
Let $x\in \mathrm{Max}(V)$ and assume, by contradiction, that $x\notin \mathcal{D}$. Then $|N[x]|\leq k-2$ and, since $N[x]$ is $n$-convex, there exists $y\in V$ such that $ N[x]\cupdot \{y\}\in n.$ As a consequence, there exists $Y\subseteq V,$ such that $N[x]\cupdot \{y\}=N[Y]$. Such a $Y$ is nonempty because $|N[x]\cupdot \{y\}|\leq k-1$ while $N[\varnothing]=V.$ Pick an element $z\in Y$. Then we have $N[z]\supseteq N[Y]=N[x]\cupdot \{y\}\supsetneq N[x]$, which implies $N[z]\supsetneq N[x]$, against the maximality of $x.$
\end{proof}

\begin{corollary}\label{Dintersect} Let $G=(V,E)$ be an $n$-convex geometry with no star. Then, for  every $K\in n\setminus\{\varnothing\},$ we have  $K\cap\mathcal{D}\neq \varnothing.$ In particular, $\{x\}\subseteq V$ is $n$-convex if and only if $x\in \mathcal{D}.$
\end{corollary}
\begin{proof}
 Let $K\in n\setminus\{\varnothing\}.$  By Lemma~\ref{upset1 third}, $K$ is an upset and hence its maximal elements belong to $\mathrm{Max}(V)$. By Proposition~\ref{Dposet},
 $\mathrm{Max}(V)=\mathcal{D}$ and thus  $K\cap\mathcal{D}\neq \varnothing.$
Now, assume that for a certain $x\in V,\{x\}\in n.$ Then $\{x\}\cap \mathcal{D}\neq \varnothing$ implies $x\in \mathcal{D}.$ Conversely, if $x\in \mathcal{D}=\mathrm{Max}(V)$,  by Lemma~\ref{upset1 third}, we deduce $\{x\}\in n.$ 
\end{proof}
\begin{prop}\label{free} Let $G=(V,E)$ be a graph with $|V|=k$. Then $n$ is the free convexity if and only if $G= K_1$, or $k\geq 2$ and $G$ is $(k-2)$-regular.
\end{prop}

\begin{proof} If $G= K_1$, then $n$ is obviously the free convexity. If $k\geq 2$  and $G$ is $(k-2)$-regular, then $k$ is necessarily even and, by Theorem~\ref{103} and Corollary~\ref{regular2}, we know that $n$ is  the free convexity. Conversely, assume that $n$ is the free convexity. Then $G$ is necessarily an $n$-convex geometry and thus we have $|\mathcal{S}|\leq 1.$ Assume first that $|\mathcal{S}|= 1.$ Let $x\in V$. Then $\{x\}\in n$. But, by Lemma~\ref{conv-star}, the only $n$-convex singleton is $\mathcal{S}$. Thus $\{x\}=\mathcal{S}$ and hence $V=\mathcal{S}$, that is, $G=K_1.$
Assume next $\mathcal{S}=\varnothing$ and pick $x\in V$. Then $\{x\}\in n$. By Corollary~\ref{Dintersect}, this implies $x\in \mathcal{D}.$ Thus $V=\mathcal{D}$ and, therefore, every vertex in $V$ has degree $k-2.$
\end{proof}

The following result  shows a relationship between convex geometries of different orders.

\begin{theorem}\label{main} Let $G=(V,E)$ be a graph with $|V|\geq 3$ which is an $n$-convex geometry without stars and let $x\in \mathcal{D}$. Then the vertex deleted subgraph $G-\{x,v(x)\}$ is an $n$-convex geometry without stars.
\end{theorem}

\begin{proof} Let $p\coloneq v(x),\  W\coloneq V\setminus\{x,p\}, \ H\coloneq G-\{x,v(x)\}=G[W].$ Note that $W\neq \varnothing$ because $|V|\geq 3,$ and hence $H$ is a graph.
  We show that $(H,n_H)$  is a convex geometry.
  We first claim that, for every $A\subseteq W,$ we have
  \begin{equation}\label{prima}
      N_G[A\cup\{x\}]=N_H[A]\cup\{x\}.
  \end{equation}
  To start with note that, by definition of the function $v$, we have $ N_G[x]=V\setminus \{p\}$ and thus  
\begin{equation}\label{fac}
      N_G[x]=W\cup\{x\}.
  \end{equation}
Let $A\subseteq W$. Note that $x\in N_G[A]$ because  $x\in N_G[W]$. Moreover, since $H$ is induced by $W,$ we also have $N_G[A]\cap W=N_H[A]$. As a consequence, using \eqref{fac}, we obtain 
$$N_G[A\cup\{x\}]=N_G[A]\cap N_G[x]=N_G[A]\cap(W\cup\{x\})=N_H[A]\cup\{x\}.$$
Thus \eqref{prima} is proven.
We next claim that 
\begin{equation}\label{seconda}
 [A]_{n_H}\cup\{x\}= [A\cup\{x\}]_{n_G}.  
\end{equation}
If $A=\varnothing, $ equality \eqref{seconda} asks for $\{x\}=[x]_{n_G}$ which is true in light of Corollary~\ref{Dintersect}.  Assume next that $A\neq \varnothing.$ Then, by Proposition~\ref{hathull}, \eqref{seconda} asks for
\begin{equation}
\label{secondabis}
 \hat{A}^H\cup\{x\}= \widehat{A\cup\{x\}^G}.  
\end{equation}
Now, using \eqref{prima} and \eqref{fac}, we have 
\begin{eqnarray}\label{ecco}
\nonumber\widehat{A\cup\{x\}^G}&=&N_G[N_G[A\cup\{x\}]]=N_G[N_H[A]\cup\{x\}]
=N_G[N_H[A]]\cap N_G[x]\\
&=&N_G[N_H[A]]\cap (W\cup\{x\}).
\end{eqnarray}
By the fact that $N_H[A]\subseteq W$ we have $x\in N_G[N_H[A].$ Moreover, recalling that $H$ is induced by $W$, by the fact that $N_H[A]\subseteq W$, we deduce 
$$N_G[N_H[A]]\cap W=N_H[N_H[A]]=\hat{A}^H.$$
Thus, from \eqref{ecco}, we obtain $\widehat{A\cup\{x\}^G}=\hat{A}^H\cup\{x\}$, proving \eqref{secondabis}. We are now ready to prove the anti-exchange property for $(H,n_H).$ Let $C\subsetneq W$ with $C\in n_H.$ We first claim that $C\cup\{x\}\in n_G.$ Indeed, 
using \eqref{seconda} with $A\coloneq C$, we have $ [C\cup\{x\}]_{n_G}=[C]_{n_H}\cup\{x\}=C\cup\{x\}.$
Let $a,b\in W\setminus C,$ with $a\neq b,$ and assume that $a\in [C\cup\{b\}]_{n_H}$. 
Then using \eqref{seconda}, with $A\coloneq C\cup\{b\},$ we deduce that $a\in [C\cup\{b,x\}]_{n_G}=[(C\cup\{x\})\cup\{b\}]_{n_G}.$ Note now that $a,b\in V\setminus (C\cup\{x\})$ with $C\cup\{x\}\in n_G$. By the anti-exchange property for $n_G$ applied to  $C\cup\{x\}$, we deduce that $b\notin [(C\cup\{x\})\cup\{a\}]_{n_G}=[C\cup\{a,x\}]_{n_G}.$ Since, by \eqref{seconda}, $[C\cup\{a,x\}]_{n_G}\supseteq [C\cup\{a\}]_{n_H}$ we deduce that $b\notin[C\cup\{a\}]_{n_H}$, concluding the proof that $(H,n_H)$ is a convex geometry.
We finally show that $H$ admits no star. Assume, by contradiction, that $\mathcal{S}(H)\neq\varnothing.$ Then, by Lemma~\ref{conv-star}, there exists a unique $s\in W$ such that $\{s\}=\mathcal{S}(H).$ Since $s\in W,$ we then have $N_G[s]\supseteq W\cup\{x\}$. We cannot have $N_G[s]\supsetneq W\cup\{x\}$ because this implies $N_G[s]=V$, against the fact that $G$ has no star. By \eqref{fac}, it follows that $N_G[s]= W\cup\{x\}=N_G[x]$.
Now, by Theorem~\ref{noclosedtwins}, we deduce $s=x\notin W,$ a contradiction.
\end{proof}


\section{Star-augmented graphs}\label{construction}

By attaching a universally adjacent vertex to a given graph, we define the notion of a star-augmented graph and this transformation preserves the property of being an $n$-convex geometry. 
We begin by defining the new construction.

\begin{defn}\label{added-star} {\rm Let $G=(V,E)$ be a graph and $s\notin V$. The graph $G^*\coloneq(V^*,E^*)$ is defined by $V^*\coloneq V\cup\{s\}$ and $E^*\coloneq E\cup\{\{v,s\}:  v\in V\}$. 
We call $G^*$ the {\bf star-augmented graph} of $G$. 
For the sake of simplicity, we will use the notation $N^*\coloneq N_{G^*}$, $\mathcal{N}^*\coloneq \mathcal{N}_{G^*}$, $n^*\coloneq n_{G^*}$ and $\operatorname{ex}^*(X)\coloneq \operatorname{ex}_{G^*}(X)$.}
    \end{defn}

Prior to the main result, we prove the following technical lemma.
\begin{lemma}
\label{lem:nconvexsets}
Let $G$ be a graph and  $G^*$ its star-augmented graph. Then the following facts hold:
\begin{itemize}
    \item [$(i)$] $n^*\subseteq\{\varnothing\}\cup \{\{s\}\cup Y:Y\in n\}.$
    \item [$(ii)$] If $G$ admits no star, then $$n^*=\{\varnothing\}\cup \{\{s\}\cup Y:Y\in n\}.$$
\end{itemize}
\end{lemma}
\begin{proof}
   $(i)$ Let $ X\in n^*-\{\varnothing\}$. Then, there exists a subset $Z\subseteq V^*$ such that $X=N^*[Z]$. If $Z\subseteq\{s\}$, then $X=V^*=V\cup\{s\}$ and $V\in n.$ Let next $Z\not \subseteq \{s\}.$ Then
    $$X=N^*[Z]=\bigcap_{x\in Z}N^*[x]=\bigcap_{x\in Z-\{s\}}(N[x]\cup\{s\})=\{s\} \cup \bigcap_{x\in Z-\{s\}}N[x] =\{s\}\cup N[Z-\{s\}].$$
    Since $N[Z-\{s\}]\in n$, we get $X\in \{\{s\}\cup Y:Y\in n\}$. It follows that $n^*\subseteq\{\varnothing\}\cup \{\{s\}\cup Y:Y\in n\}.$
\smallskip

    $(ii)$ By $(i)$, we need only to show that $\{\varnothing\}\cup \{\{s\}\cup Y:Y\in n\}\subseteq n^*.$
    The fact that $\varnothing \in n^*$ is obvious. So, let $X=\{s\}\cup Y$ for some $Y\in n.$ If $Y=\varnothing,$ since $G$ contains no star, $s$ is the only star of $G^*$ and thus $X=\{s\}=N^*[V^*]\in n^*.$ Assume next that $Y\neq \varnothing.$ Then $Y=N[U],$ for some $U\subseteq V$, and thus $X=\{s\}\cup N[U]=N^*[U]\in n^*.$
\end{proof}


\begin{theorem}
\label{th:convexgeometry}
Let $G$ be a graph with no star
and $G^*$ its star-augmented graph. Then, $G$ is an $n$-convex geometry if and only if $G^*$ is an $n$-convex geometry.
\end{theorem}
\begin{proof} By Lemma \ref{lem:nconvexsets}, we know that $n^*=\{\varnothing\}\cup \{\{s\}\cup Y:Y\in n\}.$
We start the proof claiming that when $\varnothing\neq Y\in n$ and 
$X=Y\cup\{s\}\in n^*$, then 
\begin{equation}\label{estremi}
\operatorname{ex}^*(X)=\operatorname{ex}(Y).
\end{equation}
 Let $y\in \operatorname{ex}(Y)$. Then $Y-\{y\}\in n$ and thus $X-\{y\}= (Y-\{y\})\cup \{s\}\in n^*.$ This shows $\operatorname{ex}(Y)\subseteq \operatorname{ex}^*(X).$
Let next $y\in \operatorname{ex}^*(X)$.
 Then $\varnothing\neq X-\{y\}\in n^*$ and $y\neq s$ because $s$ belongs to all the non-empty $n^*$-convex sets contained in $X$. Thus $y\in Y$ and since $X-\{y\}=(Y\cup\{s\})-\{y\}=(Y-\{y\})\cup\{s\}\in n^*$, we deduce $y\in \operatorname{ex}(Y).$ This shows $\operatorname{ex}^*(X)\subseteq \operatorname{ex}(Y).$

 Suppose now that $G$ is an $n$-convex geometry and let $X\in n^*\setminus\{\varnothing\}$. Then $X=Y\cup\{s\}$, for some $Y\in n$. 
If $Y=\varnothing,$ then $X=\{s\}$ and $\operatorname{ex}^*(X)=\{s\}$. It follows that  
$$[\operatorname{ex}^*(X)]_{n^*}=N^*[N^*[\operatorname{ex}^*(s)]]=N^*[N^*[s]]=N^*[V^*]=\mathcal{S}(G^*)=\{s\}=X.$$
Let next $X=Y\cup\{s\}$ with $Y\neq \varnothing$.
Then, using \eqref{estremi}, we have 
$$[\operatorname{ex}^*(X)]_{n^*}=N^*[N^*[\operatorname{ex}^*(X)]]=N^*[N^*[\operatorname{ex}(Y)]]=N^*[N[\operatorname{ex}(Y)]\cup\{s\}]=$$
$$=N^*[N[\operatorname{ex}(Y)]]\cap N^*[s]=N^*[N[\operatorname{ex}(Y)]]\cap V^*=N^*[N[\operatorname{ex}(Y)]]=$$
$$=N[N[\operatorname{ex}(Y)]]\cup\{s\}=Y\cup\{s\}=X.$$
So, $G^*$ is an $n$-convex geometry. 

Conversely, suppose that $G^*$ is an $n$-convex geometry and let $Y\in n\setminus\{\varnothing\}.$ Consider  $X=Y\cup\{s\}\in n^*\setminus\{\varnothing\}$. We show that $[\operatorname{ex}(Y)]_n=Y$. By \eqref{estremi}, we have $\operatorname{ex}(Y)=\operatorname{ex}^*(X)\neq\varnothing$ and hence the $n$-convex hull of $\operatorname{ex}(Y)$ coincides with its  neighbourhood closure. Thus, recalling that $s$ is a star vertex for the $n$-convex geometry $G^*$, we have 
$$[\operatorname{ex}(Y)]_n=NN[\operatorname{ex}(Y)]=NN[\operatorname{ex}^*(X)]=N[N^*[\operatorname{ex}^*(X)]-\{s\}]=$$
$$N^*[N^*[\operatorname{ex}^*(X)]-\{s\}]-\{s\}=N^*[N^*[\operatorname{ex}^*(X)]]-\{s\}=[{ex}^*(X)]_{n^*}-\{s\}=X-\{s\}=Y.$$
So, $G$ is an $n$-convex geometry.
\end{proof}

Theorem~\ref{th:convexgeometry} allows to construct an $n$-convex geometry of any vertex size.
\begin{cor}\label{allk}

    For any integer $k\geq 1$, there exists a graph with $k$ vertices which is an $n$-convex geometry.
\end{cor}
\begin{proof}
    If $k$ is even we use the $(k-2)$-regular graph $B_k$ on $k$ vertices which is an $n$-convex geometry, by Corollary~\ref{regular2}.   If $k\geq 3$ is odd, we consider the star-augmented graph $A_k$ of $B_{k-1}$. By Theorem~\ref{th:convexgeometry}, $A_k$  is an $n$-convex geometry. If $k=1,$ we have $K_1.$
\end{proof}

\section{Characterizing the existence of stars in an $n$-convex geometry}\label{sec:stars}

 We now present the main application of the results collected in the previous sections, namely, the characterization of the $n$-convex geometries admitting a star in term of their vertex set size. 



\begin{corollary}\label{evenodd}Let $G=(V,E)$ be a graph which is an $n$-convex geometry. Then $G$ admits a star if and only if $|V|$ is odd.
\end{corollary}
\begin{proof} Suppose that $G$ admits no star. Assume first that $|V|\leq 2.$ Since $K_1$ has a star, we necessarily have $|V|=2$ even and $G=2K_1$, the only $n$-convex geometry with two vertices.
Suppose next that  $|V|\geq 3$. By Theorem~\ref{stars-extreme},  $\mathcal{D}\neq \varnothing$ and we pick a vertex $x\in \mathcal{D}.$ By
Theorem~\ref{main}, $H\coloneq G-\{x,v(x)\}$ is an $n$-convex geometry without stars. Starting from this graph, if the number of vertices is at least three, we can iterate the process of deleting two vertices until we reach a graph with at most  two vertices which is an $n$-convex geometry with no star. By the case with at most two vertices examined before, we deduce that the final graph in this process admits two vertices. As a consequence
$|V|$ is even. 

What shown above, proves that if $|V|$ is odd, then $G$ must admit a star.
Conversely suppose that $G$ admits a star $s$. Then, once set $G'\coloneq G-\{s\},$ we have that $G'$ admits no star and $G$ is the star-augmented graph of $G'.$
By Theorem~\ref{th:convexgeometry}, also $G'$ is a convex geometry and it admits no star. Then, by what shown before $|V|-1$ is even, that is, $|V|$ is odd.
\end{proof}
    
We are now able to show that any $n$-convex geometry with an odd number $2k+1$ of vertices  is constructed in a precise way 
from an $n$-convex geometry with $2k$ vertices.


\begin{prop}\label{charodd}
    A graph $G$ with an odd number of vertices is an $n$-convex geometry if and only if the following conditions both hold: 
\begin{enumerate}[label=$(\arabic*)$, ref=\thetheorem (\arabic*)]
    \item \label{charodd first} $G$ contains a unique star vertex $s$;
   \item \label{charodd second} $G-\{s\}$ is an $n$-convex geometry.
\end{enumerate}
\end{prop}
\begin{proof} Let $G=(V,E)$ be a graph with $|V|$ odd. Assume that $G$ is an $n$-convex geometry. Then, by Corollary~\ref{evenodd}, $G$ admits a unique star $s$. Now consider the graph $G'\coloneq G\setminus \{s\}$. Since $G=(G')^*$, by Theorem~\ref{th:convexgeometry}, we have that $G'$ is an $n$-convex geometry. Thus $(1)$ and $(2)$ hold.

    Conversely assume that $(1)$ and $(2)$ hold. Then $G'\coloneq G\setminus \{s\}$ is an $n$-convex geometry with no star. Thus, by Theorem~\ref{th:convexgeometry}, we have that $G=(G')^*$ is an $n$-convex geometry. 
\end{proof}

\section{Applications}\label{sec:applications}
As a consequence of Theorem~\ref{333first}, taken $k \in \mathbb{N}\cup\{0\}$, there exists a finite number of graphs which are $n$-convex geometries and for which $k$ is the minimum degree. We have already seen that for $k=0$ the only ones are $K_1$ and $2K_1$. We now discover those with $k=1$. 
\begin{prop}\label{pendant}
    Let $G=(V,E)$ be a graph with at least a pendant vertex. Then $(G,n)$ is a convex geometry if and only if $G \in \{ P_3,  P_4\}$.
\end{prop}
\begin{proof}
If $G = P_3$ or $G= P_4$ then $G$ has at least a pendant vertex and we have seen in Remark~\ref{341} that, in both cases, $(G,n)$ is a convex geometry.

Suppose now that $(G,n)$ is a convex geometry and let  $v \in V$ be such that $d(v)=1$. In particular, we have $E \not = \varnothing$ and $|V| \geq 2$ and so, by Theorem~\ref{103} and~\ref{104}, $G$ is  connected and not complete. By Theorem~\ref{333}, we also have that $|V| \leq 4$. If $|V|=2$, then we have $G \cong K_2$, a contradiction. So we have $3 \leq |V| \leq 4$.
If $|V|=3$, note that the only connected graph with $3$ vertices and a pendant vertex is $P_3.$ Assume finally that $|V|=4$. There are two connected graphs with $4$ vertices and a pendant vertex, that is, $P_4$ and the so-called $3$-pan $P=([4],\{ \{1,2\}, \{2,3\}, \{3,4\}, \{1,3\}\})$. However, $P$ admits a star and hence, by Corollary~\ref{evenodd}, it is not an $n$-convex geometry. Thus $G=P_4.$
\end{proof}

We now describe some classes of graphs which are $n$-convex geometries.
\begin{prop}\label{main-examples} The following facts hold.
\begin{enumerate}[label=$(\roman*)$, ref=\thetheorem $(\roman*)$]
  \item  \label{main-examples first} The only cycle $C_n$ which is an $n$-convex geometry is $C_4$; the only paths which are $n$-convex geometries are $P_1, P_3, P_4.$ 
        \item  \label{main-examples second} The only fan which is an $n$-convex geometry is the $4$-fan.
        \item  \label{main-examples third} The only complete bipartite graphs which are $n$-convex geometries are $K_{1,2}\cong P_3$ and $K_{2,2}\cong C_4$.
\end{enumerate}
    
\end{prop}
\begin{proof}
$(i)$  Assume that $C_k$ is an $n$-convex geometry. Then, by Theorem ~\ref{333}, $\delta (G)=2\geq \frac{k-2}{2}$. As a consequence $k\leq 6$. Now $C_3$ is not an $n$-convex geometry because has three stars. $C_5$ and $C_6$ are not  $n$-convex geometries because they admit no quasi-stars as necessary by Theorem~\ref{stars-extreme}. So the only cycle which is an $n$-convex geometry is $C_4.$

Consider now  $P_k$. When $k\in\{1,3,4\}$, we know that $P_k$ is an $n$-convex geometry (see Remark~\ref{341}. Let $k\geq 5$ and assume, by contradiction, that $P_k$ is an $n$-convex geometry. Then it must admit a quasi-star, that is an element of degree $k-2\geq 3,$ a contradiction.
\smallskip

$(ii)$ Consider now the $k$-fan, that is, the graph $F_k=P_k+P_1$ for $k\geq 3$. Assume that $F_k$ is an $n$-convex geometry. Since this graph admits a star, by Corollary~\ref{evenodd}, we have that $k+1$ is odd, that is $k$ is even. We look to the vertices in $\mathcal{D}$, that is those of degree $k-1.$ Since apart the star all the vertices have degree $2$ or $3$. We must have $k-1=2$ or $k-1=3$. Since $k$ is even, the only possibility is $k=4.$

\smallskip

$(iii)$ Consider the complete bipartite graph  $G\coloneq K_{r,s}$ for $r,s\geq 1$ and $r\leq s$. Suppose that $G$ is an $n$-convex geometry. Assume first $r=1$. 
Then all the vertices apart the star are pendant. Thus, by Proposition~\ref{pendant}, $s=2.$
Assume next $r=2$. Now the graph has no star and hence $s\geq 2$ is even. If $s=2$ we find $G=K_{2,2}\cong C_4$. Assume $s\geq 4.$ 
We now count the edges. By Proposition~\ref{famiglie}, we must have $2s=|E|\geq \frac{(2+s)s}{4}$ which implies $s\leq 6.$ So, it remains to understand if $K_{2,4}$ and $K_{2,6}$ are $n$-convex geometries.
The elements in $\mathcal{D}$ are those of degree $r+s-2=s$ and they form the part $V_1=\{x_1,x_2\}$ of the graph of size $2$. Now $v(x_1)=V\setminus N[x_1]=x_2$ and, by Theorem~\ref{main}, $K_{2,s}-\{x_1,x_2\}$ is an $n$-convex geometry. But such a graph is totally disconnected with $s\geq 4$ vertices, which is impossible by Theorem~\ref{103}. 

Let now pass to consider $r\geq 3.$ In this case we also have $s\geq 3$ and $G=K_{r,s}$ fails to admit vertices of degree $r+s-2$. Indeed, the only degrees of vertices in $G$ are $r,s$. Now $r+s-2=r$ implies $s=2$; $r+s-2=s$ implies $r=2.$
\end{proof}

\subsection{Threshold and quasi-threshold $n$-convex geometries}
We introduce now a graph whose properties allow to characterize when an $n$-convex geometry is a quasi-threshold graph,  or a threshold graph. Recall that a graph is {\bf quasi-threshold} if it is $\{P_4,C_4\}$-free and {\bf threshold}  if 
it is $\{P_4,C_4, 2K_2\}$-free.

Since the list is very poor we deduce that, typically, an $n$-convex geometry admits as induced subgraph a path or a cycle on four vertices, a somewhat unexpected result. 

\begin{defn}
    {\rm The {\bf neighbourhood comparability graph} associated with $G$ is the graph $G_N=(V,E_N)$ where, for every   $x,y\in V,$ distinct $\{x,y\}\in E_N$ if $x,y$ are comparable for $\leq$, that is $ x\leq y$ or $ y\leq x$ holds.
}
   
\end{defn}
Obviously, we have that $G_N$ is a spanning subgraph of $G.$ We can easily characterize when they coincide.
\begin{lemma}\label{coco}
Let $G=(V,E)$ be a graph. Then $G=G_N$ if and only if $G$ is a quasi-threshold graph.
\end{lemma}
\begin{proof}
The proof is straightforward recalling one famous characterization of quasi-threshold graphs (see, for instance, \cite[Theorem 3]{Ya}): 
 $G$ is quasi-threshold  if and only if, for every $\{x,y\}\in E$, the neighbours $N[x]$  and $N[y]$ are comparable by inclusion. 
\end{proof}

\begin{theorem}\label{convcomp}
 Let $(G,n)$ be a convex geometry. The following facts are equivalent:
 \begin{enumerate}[label=$(\roman*)$, ref=\thetheorem $(\roman*)$]
\item \label{convcomp first} 
 $G$ is quasi-threshold;
 \item \label{convcomp second} $G\in \{K_1, 2K_1, P_3\};$ 
 \item \label{convcomp third} $G$ is threshold.
 
 \end{enumerate}In particular, if $|V|\geq 4$, then either $G$ admits an induced subgraph isomorphic to $P_4$ or isomorphic to $C_4.$ 
\end{theorem}
\begin{proof}
 $(i)\Rightarrow (ii)$ Assume that $G=(V,E)$ is a quasi-threshold graph. If $|V|=1$, then $G=K_1.$
 Let $|V|\geq 2$. Then, by Lemma~\ref{coco}, $G=G_N.$ Since $G$ is an $n$-convex geometry, by Theorem~\ref{stars-extreme}, we know that $|\mathcal{D}|\geq 2.$ 
 We claim that $|\mathcal{D}|=2.$ Assume, by contradiction, that instead $|\mathcal{D}|\geq 3.$ Let $x,y,z\in \mathcal{D} $ be distinct. Since $|V-N[x]|=1,$ one among $y,z$ belongs to $N[x].$ Say $\{y,x\}\in E.$ Then, we also have $\{y,x\}\in E_N$ and thus $N[x]$ and $N[y]$ are comparable by inclusion. However, since those sets have the same size $|V|-1$, we reach $N[x]=N[y]$ against Theorem~\ref{noclosedtwins}.
Thus $|\mathcal{D}|=2$ and, by the same argument as above, we have $\mathcal{D}=\{x_1,x_2\}$ with $x_1\neq x_2$ not adjacent. As a consequence, $v(x_1)=x_2$ and $v(x_2)=x_1.$ 
By Theorem~\ref{stars-extreme}, we now deduce  $\mathrm{ex}(V)=\mathcal{D}$.
By Proposition~\ref{Dposet} and Lemma~\ref{upset1 fifth}, we then obtain $\mathrm{Max}(V)=\mathcal{D}=\mathrm{ex}(V)\subseteq \mathrm{Min}(V).$
 Pick $x\in V$ and consider $\uparrow x.$ This upset intersects $\mathrm{Max}(V)=\mathcal{D}$ and hence we have, say, $x\leq x_1.$ Since $x_1$ is minimal in $V,$ this implies $x=x_1\in \mathcal{D}.$
As a consequence, we have
$V=\mathcal{D}$ and $G=2K_1.$ Assume next that $G$ has a star $s$. Then $V=N[(V\setminus\{x_1\})\cap (V\setminus\{x_2\})]=N[V\setminus\{x_1,x_2\}].$ It follows that
$\{s\}=N[V]=N[N[V\setminus\{x_1,x_2\}]]$. Now since $x_1, x_2$ are extreme of $V$, we have that $V\setminus\{x_1,x_2\}\in n$ and thus $\{s\}=V\setminus\{x_1,x_2\}$, that is, $V=\{s,x_1,x_2\}$. Since $x_1, x_2$ are not adjacent this forces $G=P_3.$
 \smallskip
 
 $(ii)\Rightarrow (iii)$ This is trivial.

 \smallskip
 
 $(iii)\Rightarrow (i)$ This follows immediately recalling that every threshold graph is quasi-threshold.

For the last part of the statement, simply note that when $G$ has at least $4$ vertices then it does not appear in the list $\{K_1, 2K_1, P_3\}$ and thus $G$ is not quasi-threshold. Thus it is not $\{P_4,C_4\}$-free.
\end{proof}
 Note that, as a trivial consequence of Theorem~\ref{convcomp}, a graph with four vertices is an $n$-convex geometry if and only if it is isomorphic to $P_4$ or $C_4.$ Moreover, completing the content of Proposition~\ref{famiglie second}, a forest is an $n$-convex geometry if and only if it is isomorphic to one of  $K_1, 2K_1, P_3,P_4.$
\section{Quasi-stars augmented graphs}\label{sec:final}
In this final section, we briefly report on the problem of the construction of $n$-convex geometries of even order $2k$ from those of odd order $2k-1$, for $k\geq 2$. 
We know that every $2k-1$ order $n$-convex geometry admits a star.
A natural idea is to 
introduce  arrangements similar to the one of Definition~\ref{added-star}, taking into account the presence of the star.
For the moment, we have discovered one successful arrangement, the one of adding a false twin of the star.

\begin{defn}\label{quasi-star-aug}
   {\rm  Let $G=(V,E)$ be a graph with a star vertex $s\in V$ and let $q\notin V.$ Define $G'=(V',E')$ to be the graph with $V'\coloneq V\cup\{q\}$ and $E'\coloneq E\cup\{\{x,q\}: x\in V-\{s\}\}$.  We call $G'$ the {\bf quasi-star-augmented graph} of $G$ by $q$.} 

\end{defn}
Note that $q$  is false twin of $s$ and thus $q$ is a quasi-star for $G'$. Moreover,  the original star $s$ in $G$ becomes another quasi-star in $G'$ and $G'$ admits no star.
\begin{theorem}\label{th:quasi-stars}
Let $G$ be an $n$-convex geometry with a star vertex. Then the quasi-star-augmented graph of $G$ is also an $n$-convex geometry. 
\end{theorem}

Unfortunately, unlike the odd case, not all $n$-convex geometries with of order $2k$ can be constructed applying the previous result to a suitable $n$-convex geometry of order $2k-1$. For example, we know that $P_4$ and $C_4$
are the only  $n$-convex geometries of order $4$. But there is just one  $n$-convex geometry of order $3$, that is, $P_3$ and thus just one quasi-star-augmented graph of order $4$, that is, $C_4.$ This means that a further arrangement must be discovered in order to reach also $P_4.$  

The proof of Theorem~\ref{th:quasi-stars} is technical and we omit it, for the sake of brevity. 

\section*{Conclusions  and future lines of work}
\label{sec:conclusions}

In this paper, we investigated neighbourhood convexity ($n$-convexity) in graphs, a novel convexity space derived from the common closed neighbourhood closure operator introduced in \cite{1020}. Unlike traditional path-based graph convexities, $n$-convexity is almost never hereditary, as shown by the fact that $K_1$ and $2K_1$ are the only hereditary $n$-convex geometries.  We provided structural characterizations for graphs that form $n$-convex geometries, establishing strong bounds on minimum degree, diameter, and the presence of stars. In particular, we proved a parity phenomenon: an $n$-convex geometry admits a star if and only if its vertex count is odd. Every odd-order $n$-convex geometry can be uniquely constructed as the star-augmentation of an even-order $n$-convex geometry. 
However, it remains a mystery how to construct every even-order $n$-convex geometry. Indeed, we have discovered the method of adding a quasi-star to an odd-order $n$-convex geometry, but we know that it does not suffice to construct every even-order $n$-convex geometry. A line of future research is to determine the whole family of constructions capable to provide the full list of even-order $n$-convex geometries.
By exploring the connection with the neighbourhood preorder $P(G)$, we have established that $n$-convex sets are always upsets of $P(G)$, and that quasi-stars are the maximal elements of $P(G)$ in geometries without stars. We also know that the extreme vertices of a convex set $K$ are minimal elements  of the poset induced by $P(G)$ on $K.$ We conjecture that, given a convex geometry $(G,n)$, any upset in $P(G)$ is an $n$-convex set, having checked it computationally for graphs up to 9 vertices. Moreover, we conjecture that in an $n$-convex geometry, the extreme vertices of any convex set $K$ are exactly the minimal elements  of the poset induced by $P(G)$ on $K.$ 

We finally note that we left aside some other questions as studying the algorithmic aspects of $n$-convexity, such as efficiently computing the $n$-convex hull, extreme sets, and checking whether a given graph forms an $n$-convex geometry; or the study of standard convexity parameters for $n$-convexity, including the hull number, interval number, geodetic number, and Radon-type numbers.\section*{Conflict of interest statement} 

The authors have no conflict of interest to declare.

\section*{Acknowledgements} 
 Daniela Bubboloni is supported by INdAM-GNSAGA (Italy) and by the national project PRIN 2022- 2022PSTWLB - Group Theory and Applications - CUP B53D23009410006. 
José Cáceres is supported by CDTIME and Junta de Andalucía under Grant FQM-425.
Both authors thank the people of the Séminaires du Pôle 2: Optimisation combinatoire, algorithmique, of the Lamsade research center (Paris, France), where this research has been presented in June 2026, for their comments on a preliminary version of the paper. Finally, both authors thank Francesco Mori for the kind permission to use some of the beautiful results contained in his Master's Thesis in Mathematics \cite{Mori}.


\begin{thebibliography}{99}

\bibitem{Arm} D. Armstrong, \emph{The Sorting Order on a Coxeter Group}, Journal of Combinatorial Theory, Series A 116 (2009), no. 8, pp. 1285--1305. DOI: 10.1016/j.jcta.2009.03.009.
\bibitem{BackmanDanner2024} S. Backman, R. Danner, \emph{
Convex Geometry of Building Sets},
arXiv:2403.05514, 2024.

\bibitem{bp05} D. Bertsimas, I. Popescu, \emph{Optimal inequalities in probability theory: a convex optimization approach}, SIAM Journal on Optimization, 15(3) (2005), pp. 780--804.

\bibitem{be00} J. M. Bilbao, P. H. Edelman, \emph{The Shapley value on convex geometries}, Discrete Applied Mathematics 103 (2000), pp. 33--40.

\bibitem{bst93}F. Boesch, C. Suffel, R. Tindell, \emph{The neighborhood inclusion structure of a graph}, Mathematical and  Computer Modelling, 17(11) (1993), pp. 75--28.

\bibitem{1020} D. Bubboloni, N. Pinzauti,  \emph{Critical classes of power graphs and reconstruction of directed power graphs}, Journal of Group Theory 28 (2025), pp. 713--739. https://doi.org/10.1515/jgth-2023-018.

\bibitem{bp25a} D. Bubboloni, N. Pinzauti, \emph{Critical groups and partitions of finite groups}, Mediterranean Journal of Mathematics 22(131) (2025). https://doi.org/10.1007/s00009-025-02865-8

\bibitem{co09} J. C\'{a}ceres, O. Oellermann, \emph{On $3$-Steiner simplicial orderings}, Discrete Mathematics 309 (2009), pp. 5828--5833.

\bibitem{cck24} J. Chalopin, V. Chepoi, K. Knauer, \emph{Geometry of convex geometries}, arXiv:2405.12662v2.

%

\bibitem{dp02} B. A. Davey, H. A. Priestley, Introduction to Lattices and Order, 2nd edition Cambridge University Press (2002).

\bibitem{dgkps09} M. C. Dourado, J. G. Gimbel, J. Kratochvíl, F. Protti, J. L. Szwarcfiter, \emph{On the computation of the hull number of a graph}, Discrete Math. 309 (2009), pp. 5668--5674.

\bibitem{dgpst25} M. C. Dourado, M. Gutierrez, F. Protti, R. Sampaio, S. Tondato, \emph{Characterizations of graph classes via convex geometries: A survey}, Discrete Applied Mathematics 360 (2025), pp. 246--257.

\bibitem{dnb99} F. F. Dragan, F. Nicolai, A. Brandstädt, \emph{Convexity and HHD-free graphs}, SIAM Journal on Discrete Mathematics 12 (1999), pp. 119--135.

\bibitem{d87} P. Duchet, Convexity in combinatorial structures, Proceedings of the 14th Winter School on Abstract Analysis, Publisher: Circolo Matematico di Palermo (Palermo), No.14 (1987), pp. 261--293.
\bibitem{Jamison} P. H. Edelman, R. E. Jamison, \emph{The theory of convex geometries}, Geometriae Dedicata 19 (1985),  pp. 247–270. https://doi.org/10.1007/BF00149365 
\bibitem{fj86} M. Farber, R. E. Jamison, \emph{Convexity in graphs and hypergraphs}, SIAM Journal on Algebraic Discrete Methods 7 (1986), pp. 433--444.

\bibitem{Foldes} S. Földes, P. L. Hammer,
\emph{The Dilworth number of a graph},
Annals of  Discrete Mathematics 2 (1978), 211--219.

\bibitem{KnauerTrotter2024} K. Knauer, W. T. Trotter,\emph{ 
Concepts of Dimension for Convex Geometries}, 
SIAM Journal on Discrete Mathematics 38 (2024), no. 2, pp. 1566--1585.
DOI: 10.1137/23M1559853
\bibitem{Li} L. Li, J. Wang, M. Brunetti,
\emph{Seidel matrices, Dilworth number and an eigenvalue-free interval for cographs},
Linear Algebra and its Applications 698 (2024), 56--72.
https://doi.org/10.1016/j.laa.2024.05.022.
\bibitem{ly16} D. G. Luenberger and Y. Ye, Linear and Nonlinear Programming, Springer International Publishing AG Switzerland,  2016.
\bibitem{Ma} Y. Ma, \emph{ 
Enumeration of Hopf Monoids and Supersolvable Convex Geometries},
arXiv:2506.00380, 2025.
\bibitem{Mori} F. Mori, Convexity in graphs, Master's Thesis in Mathematics (thesis supervisor  D. Bubboloni), DIMAI, University of Florence (Italy), 2024.
\bibitem{m16} K. Murota, \emph{Discrete convex analysis: A tool for economics and game theory}, The Journal of Mechanism and Institution Design, Society for the Promotion of Mechanism and Institution Design, University of York, vol. 1(1) (2016), pp. 151--273.
\bibitem{p13} I. M. Pelayo, Geodesic Convexity in Graphs, Springer-Verlag, 2013.
\bibitem{Peled}U. N. Peled, M. K. Srinivasan, \emph{Vicinal orders of trees}, Discrete Applied Mathematics 29 (1990), 211--219.
\bibitem{ps85} F. P. Preparata, M. I. Shamos, Computational Geometry: An Introduction, Springer-Verlag, New York, NY, 1985. 

\bibitem{vv93} M. Van de Vel, Theory of Convex Structures, North-Holland, Amsterdam, MA, 1993.

\bibitem{w99} D. B. West, Introduction to Graph Theory, Prentice-Hall, New Jersey, 1999.

\bibitem{Ya} J-H. Yan, J-J. Chen, G. J. Chang, \emph{Quasi-threshold graphs}, Discrete Applied Mathematics 69 (3) (1996), 247--255.
\end{thebibliography}
\end{document}